\documentclass[10pt]{amsart}
\usepackage{amssymb}
\usepackage[english]{babel}
\usepackage{fancyhdr}
\usepackage{colortbl}
\usepackage{enumerate}
\usepackage{graphicx} 
\usepackage{float}
\usepackage{mathrsfs}
\usepackage{cite}
\usepackage{kpfonts}
\usepackage{tcolorbox} 
\usepackage{enumitem}
\usepackage{bbm}
\usepackage{amsmath}
\usepackage{amsfonts}
\usepackage{amsthm} 
\usepackage{latexsym}
\usepackage{mathtools}           
\usepackage[linesnumbered, ruled, vlined]{algorithm2e}
\usepackage[left=2.3cm,right=2.3cm,top=2.65cm,bottom=2.65cm]{geometry}
\usepackage{tikz}
\usetikzlibrary{arrows,automata,shapes,calc,patterns}

\allowdisplaybreaks

\usepackage[colorlinks=true,urlcolor=green!50!black,
citecolor=red!60!black,linkcolor=blue!75!black,linktocpage,pdfpagelabels,
bookmarksnumbered,bookmarksopen]{hyperref}

\newcommand{\cS}{{\mathcal S}}

\usepackage[capitalize,nameinlink,noabbrev]{cleveref}
\newtheorem{theorem}[]{Theorem}[section]

\numberwithin{equation}{section}

\newtheorem{definition}[theorem]{Definition}

\newtheorem{lemma}[theorem]{Lemma}

\newtheorem{remark}[theorem]{Remark}

\newtheorem{proposition}[theorem]{Proposition}

\usepackage{mathtools}
\DeclarePairedDelimiter\abs{\lvert}{\rvert}
\DeclarePairedDelimiter\norm{\lVert}{\rVert}
\DeclarePairedDelimiterX{\inner}[2]{\langle}{\rangle}{#1, #2}

\newcommand{\R}{\mathbb{R}}
\newcommand{\C}{\mathbb{C}}
\newcommand{\Rp}{\mathbb{R}_{+}}
\newcommand{\N}{\mathbb{N}}
\newcommand{\Rn}{\mathbb{R}^N}

\newcommand{\intrn}{\int_{\Rn}}
\newcommand{\Lp}[1]{\mathit{L}^{#1}(\Rn)}
\newcommand{\Hs}[1]{{\mathit H}^{#1}(\Rn)}
\newcommand{\Hsr}[1]{{\mathit H}_{r}^{#1}(\Rn)}
\newcommand{\normHs}[2]{\norm{#1}_{\Hs{#2}}}
\newcommand{\normLp}[2]{\norm{#1}_{\Lp{#2}}}

\newcommand{\Ft}{\tilde{F}_{\mu}}

\renewcommand{\(}{\left(}
\renewcommand{\)}{\right)}

\renewcommand{\[}{\left[}
\renewcommand{\]}{\right]}

\newcommand{\diff}{\nabla}
\newcommand{\laplace}{\Delta}

\newcommand{\slc}{\Lambda(c)}

\newcommand{\slp}{\Lambda^{-}(c)}
\newcommand{\slm}{\Lambda^{+}(c)}
\newcommand{\slmr}{\Lambda^{+}_r(c)}
\newcommand{\slpr}{\Lambda^{-}_r(c)}

\newcommand{\spu}{s_u^{-}}
\newcommand{\smu}{s_u^{+}}

\newcommand{\var}{\varepsilon}

\usepackage{etoolbox}

\graphicspath{{Figures/}}
\usepackage{float}
\usepackage[font=small,skip=0pt]{caption}

\begin{document}
	
%\tableofcontents

%\input{Title_Abstract}
%\input{Introduction}
%\input{Preliminary}
%\input{Proofofmainproposition}
%\input{Convergence-PS-sequence}
%\input{Radial-versus-non-radial}
%\input{Strict-inequality}
%\input{Appendix}

%%%%%%%%%%%%%%%%%%%%%%%%%%%\input{Title_Abstract}%%%%%%%%%%%%%%%%%%%%%%%%%%%%%%%%%%%%%%%%%%%%%%%%%%%%%%%%%%%%%%%%%%%%%
\title[Multiple normalized solutions for a Sobolev critical  Schr\"odinger equation]{Multiple normalized solutions \\ for a Sobolev critical  Schr\"odinger equation}

\author[L. Jeanjean and T.T. Le]{Louis Jeanjean and Thanh Trung Le}

\address{
	\vspace{-0.25cm}
	\newline
	\textbf{{\small Louis Jeanjean}} 
	\newline \indent Laboratoire de Math\'{e}matiques (CNRS UMR 6623), Universit\'{e} de Bourgogne Franche-Comt\'{e}, Besan\c{c}on 25030, France}
\email{louis.jeanjean@univ-fcomte.fr}

\address{
	\vspace{-0.25cm}
	\newline
	\textbf{{\small Thanh Trung Le}} 
	\newline \indent Laboratoire de Math\'{e}matiques (CNRS UMR 6623), Universit\'{e} de Bourgogne Franche-Comt\'{e}, Besan\c{c}on 25030, France}
\email{thanh\_trung.le@univ-fcomte.fr}

\date{}
\subjclass[2010]{}
\maketitle
\begin{abstract} 
	We study the existence of standing waves, of prescribed $L^2$-norm (the {\it mass}), for the nonlinear Schr\"{o}dinger equation with mixed power nonlinearities
	\begin{align*}
	i \partial_t \phi + \Delta \phi + \mu \phi |\phi|^{q-2} + \phi |\phi|^{2^* - 2} = 0, \quad (t, x) \in \R \times \R^N, % \label{NLS0A}
	\end{align*}
	where $N \geq 3$, $\phi: \R \times \R^N \to \C$, $\mu > 0$, $2 < q < 2 + 4/N $ and $2^* = 2N/(N-2)$ is the critical Sobolev exponent. 
	It was proved in \cite{JeanjeanJendrejLeVisciglia2020} that, for small {\it mass}, ground states exist and correspond to local minima of the associated Energy functional. It was also established that despite the nonlinearity is Sobolev critical, the set of ground states is orbitally stable. Here we prove that, when $N \geq 4$, there also exist standing waves which are not ground states and are located at a mountain-pass level of the Energy functional. These solutions are unstable by blow-up in finite time. Our study is motivated by a question raised by N. Soave \cite{Soave2020Sobolevcriticalcase}.
\end{abstract}

\section{Introduction}
In this paper, we study the existence of standing waves of prescribed  $L^2$-norm (the {\it mass}), for the nonlinear Schr\"{o}dinger equation with mixed power nonlinearities
\begin{align}
i \partial_t \phi + \Delta \phi + \mu \phi |\phi|^{q-2} + \phi |\phi|^{2^* - 2} = 0, \quad (t, x) \in \R \times \Rn,  \label{NLS0}
\end{align}
where $N \geq 3$, $\phi: \R \times \Rn \to \C$, $\mu > 0$, $2 < q < 2 + \dfrac{4}{N}$ and $2^* = \dfrac{2N}{N-2}$. \\

We recall that standing waves to \eqref{NLS0} are solutions of the form $\phi(t,x) = e^{-i\lambda t}u(x), \lambda \in \R$. Then the function $u(x)$ satisfies the equation
\begin{align}
-\laplace u - \lambda u - \mu \abs{u}^{q-2} u - \abs{u}^{2^*-2} u = 0 \quad \mbox{in } \Rn. \label{eqn:Laplace}
\end{align}

When looking for solutions to \eqref{eqn:Laplace} a possible choice is to consider $\lambda \in \R$ fixed and to search for solutions as critical points of the action functional, defined on $H^1(\R^N)$ by,
$$\mathcal{A}_{\lambda, \mu}(u) :=  \dfrac{1}{2} \normLp{\diff u}{2}^2 - \dfrac{\lambda}{2} \normLp{u}{2}^2    -
\dfrac{\mu}{q} \normLp{u}{q}^q - \dfrac{1}{2^*}\normLp{u}{2^*}^{2^*}.$$
In this case one usually focuses on the existence and dynamics of minimal action solutions, namely of solutions minimizing $\mathcal{A}_{\lambda, \mu}$ among all non-trivial solutions. In that direction, the first major contribution seems to be \cite{TaoVisanZhang07}. We also refer to the recent works\cite{AkahoriIbrahimKikuchiNawa2013,AkahoriIbrahimKikuchiNawa2021,ColesGustafson2020} which concern the case where $q > 2 + 4/N$ and $\mu >0$. Note that in   \cite{KillipOhPocovnicuVisan2017} the focusing-cubic defocusing-quintic NLS in $\R^3$ is studied, see also \cite{FukayaOhta,LewinRotaNodari2020}. Finally, regarding the sole existence of minimal action solutions we refer to  \cite{AlvesSoutoMontenegro2012} where, relying on the pioneering work of Brezis-Nirenberg \cite{BrezisNirenberg1983}, the existence of positive real solutions for equations of the type of \eqref{eqn:Laplace} is addressed in a very general setting. \medskip 

Alternatively, one can search for solutions to \eqref{eqn:Laplace} having a prescribed $\mathit{L}^2$-norm. It is standard to check that the Energy functional
\begin{equation*}
%\label{L6-1}
F_{\mu}(u) = \dfrac{1}{2} \normLp{\diff u}{2}^2 - \dfrac{\mu}{q} \normLp{u}{q}^q - \dfrac{1}{2^*}\normLp{u}{2^*}^{2^*}
\end{equation*}
is of class $C^1$
and that a critical point of $F_{\mu}$ restricted to the (mass) constraint
\begin{equation*}
S(c) = \Big\{ u \in H^1(\R^N) : ||u||_{L^2(\R^N)}^2 = c \Big\}
\end{equation*}
gives a solution to \eqref{eqn:Laplace}. Here the parameter $\lambda \in \R$ arises as a Lagrange multiplier, it does depend on the solution and is not a priori given. 

The prescribed {\it mass} approach that we shall follow here,  have seen an increasing interest in these last years, applied to various related problems, see, for example, \cite{BartschJeanjeanSoave16,  BartschSoave2019,
JeanjeanLu-nonlinarity, Stefanov19} and the references within. This approach is particularly relevant from a physical point of view. Indeed, the $L^2$-norm is a preserved quantity of the evolution and the variational characterization of such solutions is often a strong help to analyze their orbital stability/instability, see, for example, \cite{BellazziniJeanjeanLuo2013,CazenaveLions1982,Soave2020,Soave2020Sobolevcriticalcase}. For future reference, we now recall, 
\begin{definition}
	We say that  $u_c \in S(c)$ is a ground state solution to \eqref{eqn:Laplace} if it is a solution having minimal Energy among all the solutions which belong to $S(c)$. Namely, if
	$$\quad F_{\mu}(u_c) =  \displaystyle \inf \big\{F_{\mu}(u), u \in S(c), \big(F_\mu\big|_{S(c)}\big)'(u) = 0 \big\}.$$
\end{definition}
Note that this definition, first introduced in \cite{BellazziniJeanjean2016} on a related problem, keeps its meaning  when the Energy $F_{\mu}$ is unbounded from below on $S(c)$. \smallskip

\cref{NLS0} can be viewed as a special case of equations of the form
\begin{align}
i \partial_t \phi + \Delta \phi + \mu \phi |\phi|^{p_1-2} + \phi |\phi|^{p_2 - 2} = 0, \quad (t, x) \in \R \times \Rn,  \label{NLS0E}
\end{align}
where it is assumed that $2< p_1 \leq p_2 \leq 2^*.$ In the study of (\ref{NLS0E}) an important r\^ole is played by the so-called $L^2$-critical exponent
$$p_c = 2 + \frac{4}{N}.$$ 
A very complete analysis of the various cases that may happen for \eqref{NLS0E}, depending on the values of $(p_1,p_2)$, has been provided recently in \cite{Soave2020,Soave2020Sobolevcriticalcase}. The paper \cite{Soave2020} deals with the cases where $p_2 < 2^*$ and \cite{Soave2020Sobolevcriticalcase} with the cases where 
$p_2 = 2^*$. Let us recall some rough elements of this study, referring to \cite{Soave2020,Soave2020Sobolevcriticalcase} for a more complete picture and for precise statements. 
If $2 < p_1 \leq p_2 < p_c$ then the associated Energy functional is bounded from below on $S(c)$ for any $c>0$. 
One speaks of a {\it mass} subcritical case and to find a ground state one looks for a global minimum on $S(c)$. The classical approach Compactness by Concentration of  P.L. Lions \cite{LIONS1984-1, LIONS1984-2} permits to treat this case.   
If $p_2 >p_c$, the functional is unbounded from below on $S(c)$ for any $c>0$ and one speaks of  a {\it mass} supercritical case. It may still be possible to find a ground state,  for example, if $p_c \leq p_1 \leq p_2 \leq 2^*$,  then a ground state exists and it is characterized as a critical point of {\it mountain-pass type} at a strictly positive level of the Energy functional. Also, if one assumes that  $2 < p_1 < p_c <p_2 \leq 2^*$,  as it is the case of (\ref{NLS0})  the presence of the lower order mass subcritical term, created, for sufficiently small values of $c>0$, a geometry of local minima on $S(c).$ 
The existence of a ground state solution, which corresponds to a local minimizer can be obtained. Finally, note that in some cases, for example if 
$2 < p_1 \leq p_2 = p_c$, the Energy functional will be bounded (or not) from below depending on the value of $c>0$, one refers to  a mass critical case. \medskip

In  \cite{JeanjeanJendrejLeVisciglia2020}, we pursued the study of \cite{Soave2020Sobolevcriticalcase}, concentrating on the case $2 < p_1 < p_c <p_2 = 2^*$, and we now recall the results obtained there.  For any 
$\mu >0$, one explicit a value $c_0 = c_0(\mu) >0$ such that, for any $c \in (0, c_0)$, there exists a set $V(c) \subset S(c)$ having the property that
\begin{equation}
\label{well}
m(c) := \inf_{u \in V(c)} F_{\mu}(u) < 0 <  \inf_{u \in \partial V(c)}F_{\mu}(u).
\end{equation}
The sets $V(c)$ and $\partial V(c)$ being given by
$$V(c) := \{ u \in S(c) : \normLp{\diff u}{2}^2  < \rho_0\}, \qquad \partial V(c) := \{ u \in S(c) : \normLp{\diff u}{2}^2  = \rho_0\}$$
for a suitable $\rho_0 >0$, depending only on $c_0 >0$ but not on $c \in (0,c_0)$.
In \cite[Theorem 1.2]{JeanjeanJendrejLeVisciglia2020} it was shown that
\begin{theorem}\label{thm-1}
	Let $N \geq 3$, $2 < q < 2 + \frac{4}{N}$. For any $\mu >0$ there exists a $c_0 = c_0(\mu) >0$ such that, for any $c \in (0, c_0)$,  $F_{\mu}$ restricted to $S(c)$
	has a ground state. This ground state is a (local) minimizer of $F_{\mu}$ in the set $V(c)$. In addition any ground state for $F_{\mu}$ on $S(c)$ is a local minimizer of $F_{\mu}$ on $V(c)$.
\end{theorem}
\begin{remark}\label{playbetweenparameter}
	The value of $c_0 = c_0(\mu) >0$ is explicit, see \cite[Lemma 2.1]{JeanjeanJendrejLeVisciglia2020}. In particular, $c_0 >0$ can be taken arbitrary large by taking $\mu >0$ small enough.
\end{remark}
\begin{remark}\label{maps}
	It was proved in \cite[Lemma 2.6]{JeanjeanJendrejLeVisciglia2020} that the map $c \mapsto m(c)$ is non-increasing and that $m(c) \to 0$ as $c \to 0$.
\end{remark}

Next, introducing  the set
\begin{align*} 
%\label{set-stable-I}
\mathcal{M}_c := \{u \in V(c) : F_{\mu}(u) = m(c)\},
\end{align*}
it was established in \cite{JeanjeanJendrejLeVisciglia2020}  that
\begin{theorem}\label{thm-2}
	Let $N \geq 3$, $2 < q < 2 + \frac{4}{N}$,  $\mu >0$ and  $c_0 = c_0(\mu) >0$ be given in \cref{thm-1}. Then, for any $c \in (0, c_0),$ the set $\mathcal{M}_c$ is compact, up to translation, and it is orbitally stable.
\end{theorem}

Notice that the definition of the orbital stability, see, for example, \cite{CazenaveLions1982} or \cite{HajaiejStuart2004}, implies the global existence 
of solutions to \eqref{NLS0} for initial datum $\varphi$ close enough to the set $\mathcal{M}_c$. Here, this fact 
is non-trivial due to the critical exponent that appears in \eqref{NLS0}, even if the $H^1(\R^N)$ norm of the solution
is uniformly bounded on the lifespan of the solution, see \cite{JeanjeanJendrejLeVisciglia2020} for more details. \medskip

This structure of local minima, for a functional which is unbounded from below, suggests the possibility to search for a solution lying at a mountain pass level. This type of solution has indeed been obtained recently on related problems, see, for example, \cite{BellazziniJeanjean2016,CingolaniJeanjean2019,NorisTavaresVerzini2019}. 
In particular, on (\ref{NLS0E}) the existence of such a mountain pass geometry had been observed in \cite{Soave2020} in a Sobolev subcritical setting, namely when $p_2 < 2^*$, and a corresponding solution had been obtained. However, when one considers the Sobolev critical case $p_2= 2^*$,  an additional difficulty arises due to the fact that to prove the existence of such a solution one needs a precise upper estimate of the associated mountain pass level. Roughly speaking this upper estimate is crucial to guarantee that a Palais-Smale sequence at the mountain pass level does not carry a bubble which, by vanishing when passing to the weak limit, would prevent its  strong convergence in $H^1(\R^N)$. \medskip

The need to obtain, in problem involving a Sobolev critical term, a sharp estimate on some minimax levels is known since the pioneering work of Brezis-Nirenberg \cite{BrezisNirenberg1983} and the usual way to derive such strict upper bound is through the use of testing functions. It will also be the case here but we shall need, in this context, to overcome non-standard difficulties due to the fact that we search for a solution with a prescribed norm. In \cite{Soave2020Sobolevcriticalcase} such difficulties were first encountered and overcome but under the assumption that $p_c \leq p_1 \leq p_2 \leq 2^*$. In that case there is no solution at an energy level below the mountain pass level.  In the problem we are considering, the need to respect $L^2$-constraint, combined with the existence of a ground state solution makes things more complex. Indeed, it appears necessary for proving the strict inequality that we need, see \eqref{eqn:2.6} in \cref{prop:4}, to control precisely the interaction between standard truncated extremal Sobolev functions, as recalled in \cref{estimates-U}, with a suitable sequence of ground states for $m(c_n)$ with $c_n \to c$, see the proof of \cref{prop:5} for more details.  Actually, the existence of a second solution to (\ref{eqn:Laplace}) was proposed in 
\cite{Soave2020Sobolevcriticalcase}  as an open problem. \medskip

From here until the end of the paper, $\mathcal{S}$ denotes the best constant in the Sobolev inequality, see \eqref{Sobolev-I}.
We now state the main result of this paper. 

\begin{theorem}\label{thm-main}
	Let $N \geq 4$, $2 < q < 2 + \frac{4}{N}$,  $\mu >0$ and  $c_0 = c_0(\mu) >0$ be given in \cref{thm-1}. Then, for any $c \in (0, c_0),$ there exists a second solution $v_c \in S(c)$ which satisfies
	$$ 0 < F_{\mu}(v_c) < m(c) + \dfrac{\mathcal{S}^{\frac{N}{2}}}{N}.$$
	In particular, $v_c \in S(c)$ is not a ground state. 
\end{theorem}
 \cref{thm-main} can be complemented in the following way.

\begin{theorem}\label{thm-additional}
Under the assumptions of \cref{thm-main} we have,
\begin{enumerate}[label=(\roman*)]
		\item \label{point:thm-additional:1} For any fixed $\mu >0$ and assuming that $c \in (0, c_0(\mu))$, 
		\begin{equation*}
		 \normLp{\diff v_c}{2}^2 \to \mathcal{S}^{\frac{N}{2}} \quad \mbox{and} \quad  F_{\mu}(v_c) \to \dfrac{\mathcal{S}^{\frac{N}{2}}}{N} \quad \mbox{as} \quad c \to 0.
		\end{equation*}
		\item \label{point:thm-additional:2}  For any fixed $c >0$, $v_c \in S(c)$ exists for any $\mu >0$ sufficiently small and 
		\begin{equation*}
		 \normLp{\diff v_c}{2}^2 \to \mathcal{S}^{\frac{N}{2}} \quad \mbox{and} \quad F_{\mu}(v_c) \to \dfrac{\mathcal{S}^{\frac{N}{2}}}{N} \quad \mbox{as}  \quad \mu \to 0.
		\end{equation*}
	\end{enumerate}
	\end{theorem}

	\begin{remark}
	\cref{thm-additional} \ref{point:thm-additional:2} can be set in parallel with \cite[Theorem 1.4 2)]{Soave2020Sobolevcriticalcase}. Note that a particular emphasis is given in \cite{Soave2020Sobolevcriticalcase} on the behavior of the solutions as $\mu \to 0$, in the spirit of the so-called Brezis-Nirenberg problem. In that direction, but for a fixed $\lambda \in \R$ problem, we also refer to \cite{ColesGustafson2020}.
	\end{remark}

\begin{theorem}\label{thm-secondary}
	Under the assumptions of \cref{thm-main} the associated standing wave $e^{- \lambda t}v_c(x)$ is strongly unstable.
\end{theorem}

We do not claim any originality in \cref{thm-secondary}. Actually this result is a direct consequence of the variational characterization of the solution obtained in  \cref{thm-main}, combined with recent advances on the subject of instability by blow-up contained in \cite{Soave2020,Soave2020Sobolevcriticalcase}. \medskip

Let us now give some elements of the strategy of the proof of \cref{thm-main}. 
We define 
\begin{align*}
%\label{def:QL}
Q_{\mu}(u) := \normLp{\diff u}{2}^2 - \mu \gamma_q \normLp{u}{q}^q - \normLp{u}{2^*}^{2^*}
\end{align*}
where
\begin{align}\label{def-gamma-q}
\gamma_q := \frac{N(q-2)}{2q}.
\end{align}
It is well known, see, for example, \cite[Lemma 2.7]{JEANJEAN1997}, that all critical points of $F_{\mu}$ restricted to $S(c)$ and thus any solution to \eqref{eqn:Laplace}  satisfies $Q_{\mu}(u) =0$. Introducing the set
\begin{align*}
\slc := \{u \in S(c): Q_{\mu}(u) = 0\}.
\end{align*}
we shall show, see \cref{lemma:14}, that it admits the decomposition into the disjoint union $\slc = \slp  \cup \slm$, where 
\begin{align}\label{def-slp-slm}
\slp := \{u \in \slc: F_{\mu}(u) <0\}, \quad \mbox{and} \quad
\slm := \{u \in \slc: F_{\mu}(u) >0\}.
\end{align}
The ground state $u_c \in S(c)$ obtained in \cref{thm-1}, see also \cite{Soave2020Sobolevcriticalcase}, lies on $\slp$ and can be characterized by
$$ F_{\mu}(u_c) = \inf_{u \in \slp} F_{\mu}(u) = \inf_{u \in V(c)} F_{\mu}(u) = m(c).$$
The critical point $v_c \in S(c)$ obtained in \cref{thm-main} will satisfy, see \cref{various-characterization},
$$ F_{\mu}(v_c) = \inf_{u \in \slm} F_{\mu}(u).$$
The proof of \cref{thm-main} will follow directly from the three propositions below.

We denote by $\Hsr{1}$ the subspace of functions in $\Hs{1}$ which are radially symmetric with respect to $0$, and we define $S_r(c) := S(c) \cap \Hsr{1}$. Accordingly, we also set $\slpr = \slp \cap \Hsr{1}$ and $\slmr = \slm \cap \Hsr{1}$.

Let 
\begin{align}\label{def:M0}
M^0(c) := \inf_{g \in \Gamma^0(c)} \max_{t \in [0,\infty)} F_{\mu}(g(t))
\end{align}
where
\begin{align*}
\Gamma^0(c) := \{g \in \mathit{C}([0,\infty), S_r(c)) :  g(0) \in \slpr, \exists t_g \mbox{ s.t. } g(t)  \in E_c \, \forall t \geq t_g\}
\end{align*}
with \begin{align*}
E_c := \{u \in S(c):  F_{\mu}(u) < 2 m(c)\} \neq \emptyset.
\end{align*}

\begin{proposition} \label{prop:3}
Let $N \geq 3$.
	For any $c \in (0, c_0)$, there exists a Palais-Smale sequence $(u_n) \subset S_r(c)$ for $F_{\mu}$ restricted to $S(c)$ at level $M^0(c)$, with $Q_\mu(u_n) \to 0$ as $n \to \infty$.
\end{proposition}

%%%%%%%%%%%%%%%%%%%%%%%%%%%%%%%%%%%%%%%%%%%%%%%%%%%%%%%%%%%%%%%%%%%%%%%%%%%%%%%%%%%%%%%
\begin{proposition} \label{prop:4}
Let $N \geq 3$.
	For any $c \in (0, c_0)$, if it holds that 
	\begin{align}\label{eqn:2.6}
	M^0(c) < m(c) + \dfrac{\mathcal{S}^{\frac{N}{2}}}{N} 
	\end{align}
	then the Palais-Smale sequence obtained in \cref{prop:3} is, up to subsequence, strongly convergent in $\Hsr{1}$. 
	%and the limit $u$ is a solution of \eqref{eqn:Laplace} with $\normLp{u}{2}^2 = c$.
\end{proposition}

%%%%%%%%%%%%%%%%%%%%%%%%%%%%%%%%%%%%%%%%%%%%%%%%%%%%%%%%%%%%%%%%%%%%%%%%%%%%%%%%%%%%%%%

\begin{proposition} \label{prop:44}
	For any $c \in (0, c_0)$, if $N \geq 4$ it holds that
	\begin{align*}
	M^0(c) < m(c) + \dfrac{\mathcal{S}^{\frac{N}{2}}}{N}.
	\end{align*}
\end{proposition}

\begin{remark}\label{R-I-1}
	If, as a consequence of Ekeland variational principle, the geometry of the mountain pass implies the existence of a Palais-Smale sequence (a PS sequence for short) at the mountain pass Energy level it is now a well-identified difficulty that such sequences may not be bounded.  To obtain a bounded PS sequence one needs to explicit a sequence having additional properties. The condition that $Q_\mu(u_n) \to 0$ as $n \to \infty$, incorporated into the variational procedure the information that any solution must satisfy the Pohozaev type identity $Q_{\mu}(u)=0$, see \cite{JEANJEAN1997} in that direction.
\end{remark}

\begin{remark}\label{R-I-2}
	To establish \cref{prop:4}, we shall make use of arguments first presented in \cite[Proposition 3.1]{Soave2020Sobolevcriticalcase}. It is important to notice that the strong convergence is only obtained by working in $H_r^1(\R^N)$. Indeed the strong convergence in $L^q(\R^N)$ of any weakly converging sequence in $H^1_{r}(\R^N)$ is crucially used.
\end{remark}

The proof of \cref{prop:44}, which is the heart of the paper, can be divided into two parts whose proofs require different types of arguments. 
Let  \begin{align*}
M(c) := \inf_{h \in \Gamma(c)} \max_{t \in [0,\infty)} F_{\mu}(h(t)) 
\end{align*}
where
\begin{align*}
\Gamma(c) := \{h \in \mathit{C}([0,\infty), S(c)) : h(0) \in V(c) \cap \{ u: F_{\mu}(u) <0\}, \exists t_h \mbox{ s.t. } h(t) \in E_c \, \forall t \geq t_h\}.
\end{align*}
\begin{proposition} \label{prop:6}
Let $N \geq 3$.
	For any $c \in (0, c_0)$,  it holds that
	\begin{align*}
	M^0(c) \leq M(c).
	\end{align*}
\end{proposition}
\begin{proposition}\label{prop:5}
	For any $c \in (0, c_0)$, if $N \geq 4$ we have that
	\begin{align}\label{belowbelow}
	M(c) < m(c) + \dfrac{\mathcal{S}^{\frac{N}{2}}}{N}.
	\end{align}
\end{proposition}

Even if the conclusion of \cref{prop:6}  may somehow been expected, the proof of this result is rather involved. Due to the fact  that the symmetric rearrangement map is not continuous from $H_+^1(\R^N)$, the subspace of non negative functions in $\Hs{1}$,
to $H_+^1(\R^N)$ if $N \geq 2$, see \cite{AlmgrenLieb-1,AlmgrenLieb-2}, it is not possible to replace a given path (of non negative functions which is not restrictive) by a path which would be a Schwarz rearrangement (elements by elements) of the initial path, see \cite[Remark 5.2]{Soave2020} for a discussion in that direction. Actually, if the strict inequality of (\ref{well}) guarantees that the functional has a mountain pass geometry, it is not a sufficient information to prove that $M^0(c) \leq M(c)$. 	A better understanding of the geometry of the functional $F_{\mu}$ is required and, for this purpose, we  introduce a set $W(c)$, directly connected with the decomposition $\Lambda(c) = \Lambda^+(c) \cup \Lambda^-(c)$ and study its relation with $V(c)$, see \cref{lemma:15}.\smallskip

Note that we need to prove that $M^0(c) \leq M(c)$ because, on one hand the compactness of the Palais-Smale sequence at the mountain pass level can only be obtained by working in $\Hsr{1}$, on the other hand to show the strict inequality in \cref{prop:5} we need to work with testing functions, testing paths actually, which are not radial.  The idea of using non-radial test functions to estimate a mountain pass level defined on a radial space seems to be new.  

\begin{remark}\label{R-I-3}
It is only in  \cref{prop:5} that appears the need to restrict ourselves to $N\geq 4$ in \cref{thm-main}, \cref{thm-additional}  and \cref{thm-secondary}. It is not clear to us if this limitation is due to the approach we have developed or if the case $N=3$ is fundamentally distinct from the case $N\geq 4$.  We believe it would be interesting to investigate in that direction.
\end{remark}

%\begin{remark}\label{R-I-4}
%We do not study in the paper the issue of the orbital stability of the mountain pass solution. Very likely, following the approach proposed in \cite{Soave2020, %Soave2020Sobolevcriticalcase} one will arrived at the conclusion that the standing wave associated to $v_c$ is unstable by blow-up.  In any case it is something which% is expect for solutions having this type of variational characterization.
%\end{remark}

The paper is organized as follows. \cref{Section2} is devoted to some preliminaries. In particular we clarify the structure of the set $\Lambda(c)$ and introduce our set $W(c)$ which will prove essential in the proof of \cref{prop:6}. The proof of \cref{prop:3}, \cref{prop:4} and \cref{prop:44} are given in \cref{Section3}, 
\cref{Section4} and  \cref{Section5} respectively.  In \cref{Section6} we give the proofs of \cref{thm-main}, \cref{thm-additional}  and \cref{thm-secondary}.  Finally \cref{Section7} presents some estimates on testing functions. \medskip

{\bf Notation :}
% and $\Hor$ for the subspace of real valued functions $\mathit{H}^1(\Rn, \R)$. 
For $p \geq 1$, the $L^p$-norm of $u \in H^1(\R^N)$  is denoted by $\normLp{u}{p}$. 
%We denote by $\Hsr{1}$ the subspace of functions in $\Hs{1}$ which are radially symmetric with respect to $0$, and we define $S_r(c) := S(c) \cap \Hsr{1}$. 
 We denote by $\Rp$ the interval $(0, \infty)$. 

{\bf Addendum :} Very recently in \cite{Wei-Wu2021}, the conclusions of \cref{thm-main} were extended to hold when $N=3$, answering thus the question raised in \cref{R-I-3}. The approach of \cite{Wei-Wu2021} differs from ours by the choice of the testing paths used to prove the strict inequality \eqref{belowbelow}. The choice done in \cite{Wei-Wu2021} is more in the spirit of the one of \cite{Tarantello92},  see  \cref{lien-tarantello}.  
%Note also that extensions of \cref{thm-additional}\ref{point:thm-additional:1} are given in \cite{Wei-Wu2021}.}

%\mycomment{I don't have any experience concerning this problem. However, I think that all Referees had read \cite{Wei-Wu2021}. I believe that if Referees agree to publish then the reason is new interesting results in our paper, this does not depend on the results in \cite{Wei-Wu2021}. I know that the results in \cite{Wei-Wu2021} are better than our results, however, we have some good lemmas that can be applied to other problems.}

%%%%%%%%%%%%%%%%%%%%%%%%%%%%%%%%%%%%%%%\input{Preliminary}%%%%%%%%%%%%%%%%%%%%%%%%%%%%%%%%%%%%%%%%%%%%%%%%%%%%%%%%%%%%%%%
\section{Preliminary results}\label{Section2}

We shall make use of the following classical inequalities :
For any $N \geq 3$  there exists an optimal constant
$\mathcal{S} > 0$ depending only on $N$, such that
\begin{align}\label{Sobolev-I}
\mathcal{S} \norm{f}_{2^*}^{2} \leq \norm{\diff f}_{2}^{2} , \qquad \forall f \in \dot{\mathit{H}}^1(\R^N), \quad \mbox{(Sobolev inequality)}
\end{align}
see \cite[Theorem IX.9]{Brezis1983}. 
If
$N \geq 2$ and $p \in [2, \frac{2N}{N-2})$ then
\begin{align}\label{Gagliardo-Nirenberg-I}
\norm{f}_{p} \leq C_{N,p} \norm{\diff f}_{2}^{\beta} \norm{f}_{2}^{(1-\beta)}, \qquad \mbox{with } \beta = N\(\dfrac{1}{2} - \dfrac{1}{p}\) \quad \mbox{(Gagliardo-Nirenberg inequality),}
\end{align}
for all  $f \in H^1(\R^N)$, see \cite{Nirenberg1959}.

%For $u \in \Hs{1}$, we denote by $|u|^{\star}$ the Schwartz rearrangement of $|u|$ (see \cite[Chapter 3]{LiebLoss2001}). It is well know that $u^{\star}$ is a positive, radially symmetric decreasing function and 
%\begin{align*}
%\normLp{|u|^{\star}}{q} = \normLp{u}{q} \mbox{ for } q\geq 1 \quad \mbox{and} \quad \normLp{\diff (|u|^{\star})}{2} \leq \normLp{\diff u}{2}.
%\end{align*}
%Therefore, we obtain that $F_{\mu}(|u|^{\star}) \leq F_{\mu}(u)$.

%%%%%%%%%%%%%%%%%%%%%%%%%%%%%%%%%%%%%%%%%%%%%%%%%%%%%%%%%%%%%%%%%%%%%%%%%%%%%%%%%%%%%%%
\begin{proposition} \label{prop:1}
Let $N \geq 3$.
	For any $\mu >0$ and any $c \in (0, c_0)$, $m(c)$ is reached by a positive, radially symmetric non-increasing function, denoted $u_c$ that satisfies, for a $\lambda_c \in \R$,
	\begin{equation} \label{eqn: 2.1}
	-\laplace u_c - \mu \abs{u_c}^{q-2} u_c - \abs{u_c}^{2^*-2} u_c = \lambda_c u_c \quad \mbox{in } \Rn.
	\end{equation}
\end{proposition}
\begin{proof}
	We recall the result in \cite{JeanjeanJendrejLeVisciglia2020} that $m(c)$ is reached at $u_0$ that satisfies
	\begin{equation*}
	-\laplace u_0 - \mu \abs{u_0}^{q-2} u_0 - \abs{u_0}^{2^*-2} u_0 = \lambda_0 u_0 \quad \mbox{in } \Rn,
	\end{equation*}
	for a $\lambda_0 \in \R$. Now, let $u_c$ be the Schwarz rearrangement of $|u_0|$. 
	%(see \cite[Chapter 3]{LiebLoss2001}). 
	Hence $u_c$ is a positive, radially symmetric non-increasing function. We also have that
	\begin{align*}
	\normLp{u_c}{2}^2 = \normLp{u_0}{2}^2 = c, \quad \normLp{\diff u_c}{2}^2 \leq \normLp{\diff u_0}{2}^2 < \rho_0 \quad \mbox{and} \quad F_\mu(u_c) \leq F_\mu (u_0).
	\end{align*}
	This implies that $u_c \in V(c)$ and hence $F_\mu(u_c) = F_\mu (u_0)$. Thus, $m(c)$ is reached by $u_c$ that satisfies \eqref{eqn: 2.1}
	for a $\lambda_c \in \R$. 
\end{proof}
%%%%%%%%%%%%%%%%%%%%%%%%%%%%%%%%%%%%%%%%%%%%%%%%%%%%%%%%%%%%%%%%%%%%%%%%%%%%%%%%%%%%%%%

%%%%%%%%%%%%%%%%%%%%%%%%%%%%%%%%%%%%%%%%%%%%%%%%%%%%%%%%%%%%%%%%%%%%%%%%%%%%%%%%%%%%%%%
%Note that we can deduce from \cref{prop:1} that the set $\mathcal{M}_c \neq \emptyset$ where
%\begin{align*}
%\mathcal{M}_c := \{u \in V(c) : F_{\mu}(u) = m(c)\}.
%\end{align*}
Now, recording that 
$$Q_{\mu}(u) =  \normLp{\diff u}{2}^2 - \mu \gamma_q \normLp{u}{q}^q - \normLp{u}{2^*}^{2^*}$$ 
we have

\begin{lemma}\label{lemma:QL}
Let $N \geq 3$.
	If $(u, \lambda) \in \Hs{1} \backslash \{0\} \times\R$ is a solution to 
	\begin{equation}\label{equa-libre}
	- \Delta u - \mu |u|^{q-2}u - |u|^{2^*-2}u = \lambda u,
	\end{equation}
	then $Q_{\mu}(u) = 0$ and $\lambda <0$.
\end{lemma}
%\mycomment{In step 2 of the proof of \cref{prop:4}, we need $\lambda \leq 0$.}
\begin{proof}
The fact that any solution to (\ref{equa-libre}) satisfies $Q_{\mu}(u) =0$ is a direct consequence of the Pohozaev identity see, for example, \cite[Lemma 2.7]{JEANJEAN1997}.  Now we deduce from (\ref{equa-libre}) that
	\begin{equation}\label{eq:free}
	\normLp{\diff u}{2}^2 - \mu  \normLp{u}{q}^q -  \normLp{u}{2^*}^{2^*} = \lambda \normLp{u}{2}^2.
	\end{equation}
	Combining (\ref{eq:free}) with $Q_{\mu}(u)=0$ we obtain that 
	$$\lambda \normLp{u}{2}^2 = - \mu(1- \gamma_q) \normLp{u}{q}^q $$
	which proves the lemma since $\gamma_q \in (0, 1).$
\end{proof}

%%%%%%%%%%%%%%%%%%%%%%%%%%%%%%%%%%%%%%%%%%%%%%%%%%%%%%%%%%%%%%%%%%%%%%%%%%%%%%%%%%%%%%%
The following lemma is implicitly contained in \cite[Lemma 2.6]{JeanjeanJendrejLeVisciglia2020} but since it will be crucial in the proof of \cref{prop:5} we now give an explicit proof.

\begin{lemma}\label{lemma:6}
Let $N \geq 3$.
	For any $c \in (0, c_0)$ there exists a $d=d(c) >0$ such that $m(c- \alpha) \leq m(c) + d \alpha$ for any $\alpha \in \displaystyle \(0, \frac{c}{2}\)$.
\end{lemma}
\begin{proof}
	Let $u \in V(c)$ be a minimizer of $m(c)$ and set
	$\displaystyle y_{\alpha} := \sqrt{\frac{c- \alpha}{c}}\cdot u$. Since $y_{\alpha} \in V(c- \alpha)$ we have
	\begin{align*}
	m(c- \alpha) \leq F_{\mu}(y_{\alpha}) = m(c) + [ F_{\mu}(y_{\alpha}) - F_{\mu}(u)]  
	\end{align*}
	with
	\begin{equation*}
	F_{\mu}(y_{\alpha}) - F_{\mu}(u) = - \frac{\alpha}{2c}\normLp{\diff u}{2}^2 - \frac{\mu}{q} \[ \(\frac{c- \alpha}{c}\)^q - 1 \]  \normLp{u}{q}^q -  \frac{1}{2^*} \[ \(\frac{c- \alpha}{c}\)^{2^*} - 1 \] \normLp{u}{2^*}^{2^*}.
	\end{equation*}
	The result now follows from the observation that $\alpha \mapsto F_{\mu}(y_{\alpha}) - F_{\mu}(u)$ is of class $C^1$ on  $ \displaystyle \(0, \frac{c}{2}\)$.
\end{proof}
%%%%%%%%%%%%%%%%%%%%%%%%%%%%%%%%%%%%%%%%%%%%%%%%%%%%%%%%%%%%%%%%%%%%%%%%%%%%%%%%%%%%%%%

Now let $u \in S(c)$ be arbitrary but fixed. For $s \in \Rp$ we set
\begin{align*}
u_s(x) := s^{\frac{N}{2}} u(s x).
\end{align*}
Clearly $u_s \in S(c)$ for any $s \in \Rp$. We define on $\Rp$ the fiber map,
\begin{align}\label{eqn:2.4}
\psi_u(s) := F_{\mu} (u_s) = \dfrac{s^2}{2} \normLp{\diff u}{2}^2 - \dfrac{\mu}{q} s^{q\gamma_q} \normLp{u}{q}^q - \dfrac{s^{2^*}}{2^*} \normLp{u}{2^*}^{2^*},
\end{align}
where $\gamma_q$ is given in  \eqref{def-gamma-q}.
Note that $\gamma_q \in (0, 1)$ and $q\gamma_q \in (0,2)$. 
We also have
\begin{align*}
\psi'_u(s) = s \normLp{\diff u}{2}^2 - \mu \gamma_q s^{q\gamma_q - 1} \normLp{u}{q}^q - s^{2^*-1}\normLp{u}{2^*}^{2^*} = \dfrac{1}{s} Q_{\mu}(u_s).
\end{align*}

%%%%%%%%%%%%%%%%%%%%%%%%%%%%%%%%%%%%%%%%%%%%%%%%%%%%%%%%%%%%%%%%%%%%%%%%%%%%

\begin{lemma} \label{lemma:14}
Let $N \geq 3$
	and $c \in (0, c_0)$. For every $u \in S(c)$, the function $\psi_u$ has exactly two critical points $\spu$ and $\smu$ with $0< \spu < \smu$. Moreover:
	\begin{enumerate}[label=(\roman*)]
		\item \label{point:lemma14:1} $\spu$ is a local minimum point for $\psi_u$, $F_{\mu}(u_{\spu}) <0$  and $u_{\spu} \in V(c)$.
		\item \label{point:lemma14:2} $\smu$ is a global maximum point for $\psi_u$, $\psi_u'(s) < 0, $ for all $s > s_u^+$ and $$F_{\mu}(u_{\smu}) \geq  \displaystyle \inf_{u \in \partial V(c)} F_{\mu}(u) > 0.$$
		\item \label{point:lemma14:3} $\psi_u''(\smu) < 0$ and  the map $u \in S(c) \mapsto s_u^+ \in \R$ is of class $C^1.$
	\end{enumerate}
\end{lemma}

\begin{proof}
	Let $u \in S(c)$ be arbitrary. Since $\psi_u(s) \to 0^-$, $\normLp{\nabla u_s}{2} \to 0$, as $s \to 0$ and $\psi_u(s) = F_{\mu}(u_s) > 0$ when $u_s \in \partial V(c) = \{v \in S(c) : \normLp{\nabla v}{2}^2 = \rho_0 \}$, necessarily $\psi_u'$ has a first zero 
	$\spu>0$ corresponding to a local minima. In particular,  $u_{\spu} \in V(c)$ and $F(u_{\spu}) = \psi_u(\spu) <0$. Now, from $ \psi_u(\spu) <0,$ $\psi_u(s) >  0$ when $u_s \in \partial V(c)$ and $\psi_u(s) \to - \infty$ as $s \to  \infty$,  $\psi'_u$ has a second zero $\smu > \spu$ corresponding to a local maxima of $\psi_u$ with  $F_{\mu}(u_{\smu}) \geq   \inf_{u \in \partial V(c)} F_{\mu}(u) > 0$.
	
	To conclude the proofs of \ref{point:lemma14:1} and \ref{point:lemma14:2}, it just suffices to show that $\psi_u'$ has at most two zeros.  However, this is equivalent to showing that the function
	$$s \mapsto \frac{\psi'_u(s)}{s}$$
	has at most two zeros. We have
	\begin{align*}
	\theta(s) := \dfrac{\psi_u'(s)}{s} = \normLp{\diff u}{2}^2 - \mu \gamma_q s^{q\gamma_q - 2} \normLp{u}{q}^q - s^{2^*-2}\normLp{u}{2^*}^{2^*}
	\end{align*}
	and
	\begin{align*}
	\theta'(s) = - \mu (q\gamma_q - 2) \gamma_q s^{q\gamma_q - 3} \normLp{u}{q}^q - (2^*-2) s^{2^*-3}\normLp{u}{2^*}^{2^*}.
	\end{align*}
	Since $q\gamma_q - 2 < 0$ and $2^*-2 > 0$, the equation $\theta'(s) = 0$ has a unique solution and hence $\theta(s)$ has indeed at most two zero points. 
	
	To establish \ref{point:lemma14:3} let us first show that $\psi_u''(s_u^+) <0$. In this aim, first note that in view of \ref{point:lemma14:1} and \ref{point:lemma14:2}, $\psi_u''(s)$ has a zero 
	$s_u^0 \in (s_u^-, s_u^+).$ Now, by direct calculations
	$$ \psi_u''(s) = \normLp{\diff u}{2}^2  - \mu \gamma_q (q \gamma_q -1) s^{q\gamma_q - 2} \normLp{u}{q}^q - (2^*-1) s^{2^*-2}\normLp{u}{2^*}^{2^*}.$$
	We distinguish two cases. If $q \gamma_q -1 \leq 0$ then $\psi_u''(s)$ has at most one zero and we are done. If $q \gamma_q -1 > 0$, then, knowing that $\psi_u''(s)$ has a zero we deduce that $\psi_u''(s)$ has exactly two zeros that we denote by $s_u^1 < s_u^2$. To conclude it suffices to show that $s_u^0 = s_u^2$ since this would imply that $s_u^+$ cannot be a zero of $\psi_u''(s)$. To show this, we assume by contradiction that $s_u^0 = s_u^1$. Then, since $\psi_u''(s) <0$ for $s \in (0, s_u^1)$ and recording that $\psi_u'(s) <0$ for $s>0$ small we deduce that $\psi_u'(s) <0$ for $s \in (0, s_u^0).$ This contradicts the fact that $\spu < s_u^0$ satisfies $\psi'(\spu) =0$. At this point we have proved that $\psi_u''(s_u^+) <0$. Now \ref{point:lemma14:3} follows from a direct application of the Implicit Function Theorem to the $C^1$ function $\Psi : \R \times S(c) \mapsto \R$ defined by $\Psi(s,u) = \psi_u'(s),$ taking into account that $\Psi(s_u^+, u) =0$ and
	$\partial_s \Psi(s_u^+, u) = \psi_u''(s_u^+) < 0.$
\end{proof}

In view of \cref{lemma:14}, the set $\slc := \{u \in S(c): Q_{\mu}(u) = 0\} $  admits the decomposition into the disjoint union $\slc = \slp  \cup \slm$, see \eqref{def-slp-slm} for the definitions of $\slp$ and $\slm$. 

%where 
%\begin{align*}
%\slp &:= \{u \in \slc: F_{\mu}(u) <0\} = \{u \in S(c): \psi_u'(1) = 0,  F_{\mu}(u) <0\}, \\
%\slm &:= \{u \in \slc: F_{\mu}(u) >0\} = \{u \in S(c): \psi_u'(1) = 0,  F_{\mu}(u) >0\}.
%\end{align*}
%\mycomment{We introduced sets $\slp$ and $\slm$ in page 4. So, I think that we should drop above texts.}

\begin{lemma} \label{lemma:15}
Let $N \geq 3$. Introducing, for any $c \in (0, c_0)$, the set
	\begin{align*}
	W(c) := \{u \in S(c): \smu > 1\}
	\end{align*}
	it holds that
	\begin{enumerate}[label=(\roman*)]
		\item \label{point:lemma15.1} $\slp \subset W(c)$.
		\item \label{point:lemma15.2}  $\partial W(c) = \slm$ and $\displaystyle \inf_{u \in \partial W(c)} F_{\mu}(u) >0.$
		\item \label{point:lemma15.3} $V(c) \cap \{u : F_{\mu}(u) <0 \} \subset W(c).$ 
		\item \label{point:lemma15.4} $\displaystyle \inf_{u \in W(c)} F_{\mu}(u)$ is reached and $\displaystyle \inf_{u \in W(c)} F_{\mu}(u) = m(c)$.
	\end{enumerate} 
\end{lemma}
\begin{proof}
	Points \ref{point:lemma15.1} and \ref{point:lemma15.2} are direct consequence of \cref{lemma:14} and of the definition of $W(c)$. To prove \ref{point:lemma15.3} we assume by contradiction that there exists a $v \in V(c) \cap \{u : F_{\mu}(u) <0 \} $ with $v \not \in W(c)$. Since $v \not \in W(c)$, then by \cref{lemma:14} \ref{point:lemma14:2} we know that $\psi_v'(s) <0$ for all $s \geq 1$. Thus, for all $s \geq 1$,
	\begin{equation}\label{nice}
	F_{\mu}(v_s) = \psi_v(s) \leq \psi_v(1) = F_{\mu}(v) <0.
	\end{equation}
	But, since $v \in V(c)$ there exists a $s_0 >1$ such that $v_{s_0} \in \partial V(c)$. Recording that $F_{\mu}(u) >0$ for any $u \in \partial V(c)$ we get a contradiction with (\ref{nice}). This proves \ref{point:lemma15.3}. Now, still in view of  \cref{lemma:14} and the definition of $W(c)$ we have that
	\begin{equation}\label{L1}
	\inf_{u \in W(c)} F_{\mu}(u) = \inf_{u \in \slp} F_{\mu}(u).
	\end{equation}
	Also, we know from \cite[Lemma 2.4]{JeanjeanJendrejLeVisciglia2020} that $\slp \subset V(c)$ and since any minimizer for $F_{\mu}$ on $V(c)$ must belong to $\slp$ it follows that 
	\begin{equation}\label{L2}
	m(c)= \inf_{u \in V(c)} F_{\mu}(u) = \inf_{u \in \slp} F_{\mu}(u).
	\end{equation}
	Gathering \eqref{L1} and \eqref{L2} and recording from \cite{JeanjeanJendrejLeVisciglia2020} that $m(c)$ is reached the conclusion follows.
\end{proof}

\section{The proof of \cref{prop:3}}\label{Section3}

We follow the strategy introduced in \cite{JEANJEAN1997} and consider the functional $\Ft : \Rp \times \Hs{1} \to \R$ defined by
\begin{align*}
\Ft(s, u) := F_{\mu} (u_s) = \psi_u(s) = \dfrac{s^2}{2} \normLp{\diff u}{2}^2 - \dfrac{\mu}{q} s^{q\gamma_q} \normLp{u}{q}^q - \dfrac{s^{2^*}}{2^*} \normLp{u}{2^*}^{2^*}. 
%\label{eqn:2.4}
\end{align*}
Note that
\begin{align}
\partial_s \Ft(s, u) = \psi_u'(s) = \dfrac{1}{s} Q_{\mu} (u_s) \label{eqn:2.9}
\end{align}
and, for any $v \in \Hs{1}$,
\begin{align}
\begin{split}
\partial_u \Ft(s, u)(v) &= s^2 \intrn \diff u \diff v dx - \mu s^{q\gamma_q} \intrn |u|^{q-2} u v dx - s^{2^*} \intrn |u|^{2^*-2} u v dx = F_{\mu}'(u_s) (v_s).
\end{split} \label{eqn:2.13}
\end{align}

We recall that the tangent space at a point $u \in S(c)$ is defined as
\begin{align*}
T_uS(c) = \{v \in \Hs{1}: \inner{u}{v}_{\Lp{2}} = 0\},
\end{align*}
and that, for any $u \in S(c)$ and any $v \in T_uS(c)$, 
\begin{align}
\inner{{F_{\mu}}'_{|S(c)}}{v} = \inner{F_{\mu}'}{v}. \label{eqn:2.15}
\end{align}

\begin{lemma}\label{lemma:5}
Let $N \geq 3$.
	For $u \in S(c)$ and $s > 0$, the map
	\begin{align*}
	T_uS(c) \to T_{u_s}S(c), \quad \phi \mapsto  \phi_s
	\end{align*}
	is a linear isomorphism with inverse $$T_{u_s}S(c) \to T_uS(c), \quad \psi \mapsto \psi_{\frac{1}{s}}.$$
\end{lemma}
%\mycomment{I also understand the point 3) of the Referee 2. I only write this lemma more detail.}
\begin{proof}
	We follow the approach in \cite[Lemma 3.6]{BartschSoave2019}. For $\phi \in T_uS(c)$ and for $t > 0$, we have
	\begin{align*}
	\intrn u_t(x) \phi_t(x) dx = \intrn t^N u(tx) \phi(tx) dx = \intrn u(y) \phi(y) dy = 0.
	\end{align*}
	As a consequence, $\phi_t \in T_{u_t} S(c)$ and the map is well defined. Clearly it is linear. Taking into account that, for every $t, s > 0$ and $w \in \Hs{1}$, 
	\begin{align*}
	w_{ts} = (ts)^{\frac{N}{2}} w(ts x) = (w_t)_s,	
	\end{align*}
	we obtain that the map is linear isomorphism.
\end{proof}
%%%%%%%%%%%%%%%%%%%%%%%%%%%%%%%%%%%%%%%%%%%%%%%%%%%%%%%%%%%%%%%%%%%%%%%%%%%%%%%%%%%%%%%

\begin{definition} \label{def:2}
	Given $c > 0$, we say that $\Ft$ has a mountain pass geometry on $\Rp \times S_r(c)$ at level $\tilde{M}(c)$ if
	\begin{align*}
	\tilde{M}(c) := \inf_{\tilde{h} \in \tilde{\Gamma}(c)} \max_{t \in [0,\infty)} \Ft(\tilde{h}(t)) > \max \{\Ft(\tilde{h}(0)), \Ft(\tilde{h}(t_{\tilde{h}}))\}
	\end{align*}
	where
	\begin{align*}
	\tilde{\Gamma}(c) := \{\tilde{h} \in \mathit{C}([0,\infty), \Rp \times S_r(c)) : \tilde{h}(0) \in (1, \slpr), \exists t_{\tilde{h}} > 0 \mbox{ s.t. } \tilde{h}(t) \in (1, E_c) \, \forall t \geq t_{\tilde{h}}\}.
	\end{align*}
\end{definition}

%%%%%%%%%%%%%%%%%%%%%%%%%%%%%%%%%%%%%%%%%%%%%%%%%%%%%%%%%%%%%%%%%%%%%%%%%%%%%%%%%%%%%%%
Recording that the definition of $M^0(c)$ is given in (\ref{def:M0}), we have,
\begin{lemma} \label{lemma:3}
Let $N \geq 3$.
	For any $c \in (0, c_0)$,  $\Ft$ has a mountain pass geometry at the level $\tilde{M}(c)$. Moreover, $M^0(c) = \tilde{M}(c)$.
\end{lemma}
\begin{proof}
	Let $h \in \Gamma^0(r)$, since $\tilde{h}(t) = (1, h(t)) \in \tilde{\Gamma}(c)$ and $\Ft(\tilde{h}(t)) = F_{\mu}(h(t))$ for all $t \in \Rp$, we have that
	$M^0(c) \geq \tilde{M}(c)$.
	Next, we shall prove that $\tilde{M}(c) \geq M^0(c)$. For all $\tilde{h}(t) = (s(t), v(t)) \in \tilde{\Gamma}(c)$, we have $s(0) = 1$, $v(0) \in \slpr$ and there exists a $t_{\tilde{h}} > 0$ such that $s(t) = 1$, $v(t) \in E_c$ for all $t \geq t_{\tilde{h}}$.
	Setting $h(t) = v(t)_{s(t)}$, we have that $h$ is continuous from $[0, \infty)$ into $S_r(c)$ and 
	\begin{align*}
	h(0) = v(0)_{s(0)} = v(0) \in \slpr, \quad h(t) = v(t)_{s(t)} = v(t) \in E_c \quad \forall t \geq t_{\tilde{h}}.
	\end{align*}
	Hence, $h \in \Gamma^0(r)$ and $\Ft(\tilde{h}(t)) = F_{\mu}(v(t)_{s(t)}) = F_{\mu}(h(t))$. Thus, $\tilde{M}(c) \geq M^0(c)$ and finally $M^0(c) = \tilde{M}(c)$.
	
	Now we claim that
	\begin{equation}\label{MPpositive}
	M^0(c) >0 \quad \mbox{for any} \quad  c \in (0, c_0).
	\end{equation}
	Indeed, let $g \in \Gamma^0(c)$ be arbitrary. Since $g(0) \subset \slpr $ in particular $g(0) \in V(c)$. Now for $t >0$ large, since $F_{\mu}(g(t)) < 2 m(c)$, necessarily in view of (\ref{well}), $g(t) \not \in V(c)$. By continuity of $g$ there exists a $t_0 >0$ such that $g(t_0) \in \partial V(c)$ and using again (\ref{well}) we conclude. \medskip
	
	At this point observing that 
	\begin{align*}
	\max\{ \Ft(\tilde{h}(0)), \Ft(\tilde{h}(t_h)) \} = \max\{ F_{\mu}(h(0)), F_{\mu}(h(t_h)) \} <0
	\end{align*}
	it follows that $\Ft$ has a mountain pass geometry at level $\tilde{M}(c)$ for all $0 < c < c_0$.
\end{proof}
%%%%%%%%%%%%%%%%%%%%%%%%%%%%%%%%%%%%%%%%%%%%%%%%%%%%%%%%%%%%%%%%%%%%%%%%%%%%%%%%%%%%%%%

\begin{proof}[Proof of \cref{prop:3}]
	Following \cite[Section 5]{Ghoussoub1993}, we set
	\begin{enumerate}
		\item $\mathcal{F} = \big\{\tilde{h}([0,\infty)): \tilde{h} \in \tilde{\Gamma}(c) \big\}$.
		\item $B = (1, \slpr) \cup (1, E_c)$.
		\item $F = \{(s,u) \in \Rp \times S_r(c): \Ft(s, u) \geq \tilde{M}(c)\}$.
	\end{enumerate} 
	Since $\Ft$ has a  mountain pass geometry at level $\tilde{M}(c)$ (see \cref{lemma:3}) and by the definition of the superlevel set $F$, we obtain $F \setminus B = F$ and 
	\begin{align}
	\sup_{(1, u) \in B} \Ft(s, u) \leq \tilde{M}(c) \leq \inf_{(s, u) \in F} \Ft(s, u). \label{eqn:2.7}
	\end{align}
	For any $A \in \mathcal{F}$, there exists a $h_0 \in \tilde{\Gamma}(c)$ such that $A = h_0([0,\infty))$ and
	\begin{align*}
	\tilde{M}(c) = \inf_{\tilde{h} \in \tilde{\Gamma}(c)} \max_{t \in [0,\infty)} \Ft(\tilde{h}(t)) \leq \max_{t \in [0,\infty)} \Ft(h_0(t)).
	\end{align*}
	Hence, there exists a $t_0 \in [0,\infty)$ such that $\tilde{M}(c) \leq \Ft(h_0(t_0))$. This means that $h_0(t_0) \in F$ and consequently,
	\begin{align}
	A \cap F \setminus B \neq \emptyset, \qquad \forall A \in \mathcal{F}. \label{eqn:2.8}
	\end{align}
	Now, for all $(s, u) \in \Rp \times S_r(c)$, we have
	\begin{align*}
	\Ft(s, u) = F_{\mu}({u}_{s}) = F_{\mu}(\abs{u}_{s}) = \Ft(1, \abs{u}_{s}).
	\end{align*}
	Hence, for any minimizing sequence $(z_n = (\alpha_n, \beta_n) ) \subset \tilde{\Gamma}(c)$ for $\tilde{M}(c)$, the sequence $(y_n = (1, \abs{\beta_n}_{\alpha_n}) )$ is also a minimizing sequence for $\tilde{M}(c)$.
	
	Using the terminology in \cite[Section 5]{Ghoussoub1993}, it means that $\mathcal{F}$ is a homotopy stable family of compact subset of $\R \times S_r(c)$ with extended closed boundary $B$ and the superlevel set $F$ is a dual set for $\mathcal{F}$. By \eqref{eqn:2.7} and \eqref{eqn:2.8}, we can apply \cite[Theorem 5.2]{Ghoussoub1993} with the minimizing sequence $\{y_n = (1, \abs{\beta_n}_{\alpha_n}) \}$. This implies that there exists a Palais-Smale sequence ${(s_n, w_n)} \subset \Rp \times S_r(c)$ for $\Ft$ restricted to $\Rp \times S_r(c)$ at level $\tilde{M}(c)$, that is, as $n \to \infty$,
	\begin{align}
	\partial_s \Ft(s_n, w_n) \to 0, \label{eqn:2.10}
	\end{align}
	and 
	\begin{align}
	\norm{\partial_u \Ft(s_n, w_n) }_{(T_{w_n}S(r))^*} \to 0, \label{eqn:2.11}
	\end{align}
	with the additional property that
	\begin{align}
	\abs*{s_n - 1} + \normHs{w_n - \abs{\beta_n}_{\alpha_n}([0,\infty))}{1} \to 0. \label{eqn:2.12}
	\end{align}
	By \eqref{eqn:2.9}, \eqref{eqn:2.10} and since $(s_n)$ is bounded due to \eqref{eqn:2.12}, we obtain $Q_{\mu}((w_n)_{s_n}) \to 0$ as $n \to \infty$.  Also, by \eqref{eqn:2.13}, the condition \eqref{eqn:2.11} implies that 
	\begin{align}
	F_{\mu}'((w_n)_{s_n}) ((\phi)_{s_n}) \to 0, \label{eqn:2.14}
	\end{align}
	as $n \to \infty$, for every $\phi \in T_{w_n}S_r(c)$. Let then $u_n:= (w_n)_{s_n}$. By \eqref{eqn:2.15}, \eqref{eqn:2.14} and \cref{lemma:5}, we obtain that $(u_n) \subset S_r(c)$ is a Palais-Smale sequence for $F_{\mu}$ restricted to $S_r(c)$ at level $M^0(c)$, with $Q_{\mu}(u_n) \to 0$. Since the problem is invariant under rotations, $(u_n) \subset S_r(c)$ is also the Palais-Smale sequence for $F_{\mu}$ restricted to $S(c)$ at level $M^0(c)$, with $Q_{\mu}(u_n) \to 0$.
\end{proof}
%%%%%%%%%%%%%%%%%%%%%%%%%%%%%%%%%%%%%%%%%%%%%%%%%%%%%%%%%%%%%%%%%%%%%%%%%%%%%%%%%%%%%%%

%%%%%%%%%%%%%%%%%%%%%%%%%%%%%%%%%%%%%%%%%%%%\input{Convergence-PS-sequence}%%%%%%%%%%%%%%%%%%%%%%%%%%%%%%%%%%%%%%%%%%%%%%%%

\section{The proof of \cref{prop:4}} \label{Section4}

In this section we give the

\begin{proof}[Proof of \cref{prop:4}]
	
	Let $(u_n) \subset \Hsr{1}$ be given by \cref{prop:3}. To show its convergence we proceed in three steps.
	
	\medbreak
	\noindent \textbf{Step 1: }  \textit{$(u_n) \subset \Hsr{1}$} is bounded.
	\medbreak
	
	Since $Q_{\mu}(u_n) \to 0,$ we have, using the Gagliardo-Nirenberg inequality (\ref{Gagliardo-Nirenberg-I}),
	\begin{align*}
	F_{\mu}(u_n) &= \dfrac{1}{N} \normLp{\diff u_n}{2}^2 - \dfrac{\mu}{q} \Big( 1 - \frac{q \gamma_q}{2^*} \Big) \normLp{u_n}{q}^q + o_n(1) \\
	&\geq \dfrac{1}{N} \normLp{\diff u_n}{2}^2 - \dfrac{\mu}{q} C_{N,q}^q \Big( 1 - \frac{q \gamma_q}{2^*} \Big) c^{(1- \gamma_q)q} \, \normLp{\diff u_n}{2}^{q \gamma_q} + o_n(1),
	\end{align*} 
	where $o_n(1) \to 0$ as $n \to \infty.$ Since $F_{\mu}(u_n) \to M^0(c) < \infty$ and $q \gamma_q <2$ the conclusion follows.
	
	\medbreak
	\noindent \textbf{Step 2: } \textit{$(u_n) \subset \Hsr{1}$ has a non-trivial weak limit.}
	\medbreak
	
	Since $(u_n) \subset \Hsr{1}$ is a bounded sequence, by the compact embedding of $\Hsr{1}$ into $\Lp{q}$, there exists a  $u \in \Hsr{1}$  such that, up to a subsequence, $u_n \rightharpoonup u$ weakly in $\Hsr{1}$, $u_n \to u$ strongly in $\Lp{q}$. \medskip

	Let us assume now, by contradiction, that $u$ is trivial. Then,  $\normLp{u_n}{q} \to 0$ and since $Q_{\mu}(u_n) \to 0$, using the Sobolev embedding, see 
	\eqref{Sobolev-I},
	we deduce that
	\begin{align}\label{E4}
	\mathcal{S}\normLp{u_n}{2^*}^{2} \leq \normLp{\diff u_n}{2}^2 \leq \normLp{u_n}{2^*}^{2^*} + o_n(1).
	\end{align}
	We distinguish the two cases 
	$$ \mbox{either} \quad \mbox{(i)} \quad \normLp{u_n}{2^*}^{2^*} \to 0 \quad \mbox{or} \quad \mbox{(ii)} \quad \normLp{u_n}{2^*}^{2^*} \to \ell >0.$$ 
	If (i) holds  then, in view of \eqref{E4}, we also have that $\normLp{\diff u_n}{2}^2 \to 0$ which implies that $F_{\mu}(u_n) \to 0$ contradicting the fact that $M^0(c) >0$, see \eqref{MPpositive}. If (ii) holds we deduce from \eqref{E4} that
	\begin{align*}
	%\label{E5}
	\normLp{u_n}{2^*}^{2^*} \geq \mathcal{S}^{\frac{N}{2}} + o_n(1)
	\end{align*}
	and thus, recording that $Q_{\mu}(u_n) \to 0$ and $\normLp{u_n}{q} \to 0$, it follows that
	\begin{align}\label{E6}
	\normLp{\diff u_n}{2}^2 =  \normLp{u_n}{2^*}^{2^*} + o_n(1) \geq \mathcal{S}^{\frac{N}{2}} + o_n(1).
	\end{align}
	From \eqref{E6} we deduce that
	$$F_{\mu}(u_n) = \frac{1}{N} \normLp{\diff u_n}{2}^2  + o_n(1) \geq \frac{1}{N}\mathcal{S}^{\frac{N}{2}} + o_n(1).$$
	But, since $m(c) <0$,  necessarily $ M^0(c) < \dfrac{\mathcal{S}^{\frac{N}{2}}}{N}$ and we also have a contradiction. 
	\medbreak
	\noindent \textbf{Step 3: } \textit{$(u_n) \subset \Hsr{1}$ strongly converges.}
	\medbreak
	
	Since $(u_n)$ is bounded, following 
	\cite[Lemma 3]{BerestyckiLions1983-2}, we know that
	\begin{align*}
	\big(F_\mu\big|_{S(c)}\big)'(u_n) \to 0 \text{ in } \Hs{-1}
	\iff F_{\mu}'(u_n) - \dfrac{1}{c}\inner{F_{\mu}'(u_n)}{u_n}u_n \to 0 \text{ in } \Hs{-1}.
	\end{align*}
	Thus, for any $w \in \Hs{1}$, we have
	\begin{align}\label{E1}
	o_n(1) &= \inner*{F_{\mu}'(u_n) - \dfrac{1}{c}\inner{F_{\mu}'(u_n)}{u_n}u_n}{w} =  \int_{\R^N} \nabla u_n \cdot \nabla w  - \mu |u_n|^{q-2}u_n w - |u_n|^{2^* -2}u_n w - \lambda_n u_n w \, dx,
	\end{align}
	where $o_n(1) \to 0$ as $n \to \infty$ and 
	\begin{align}\label{E2}
	c \,\lambda_n = \normLp{\diff u_n}{2}^2 - \mu \normLp{u_n}{q}^q - \normLp{u_n}{2^*}^{2^*} + o_n(1).
	\end{align}
	In particular $(\lambda_n) \subset \R$ is bounded and, up to a subsequence, $\lambda_n \to \lambda \in \R$. Now, passing to the limit in \eqref{E1} by weak convergence, we obtain that
	\begin{align}\label{E3}
	- \Delta u - \mu |u|^{q-2}u - |u|^{2^*-2}u =  \lambda u.
	\end{align}
	Thus in view of \cref{lemma:QL}, $Q_{\mu}(u)=0$ and $\lambda < 0$.

	Let $(v_n) \subset \Hsr{1}$ be such that $v_n = u_n - u$.  We have that $v_n \rightharpoonup 0$ weakly in $\Hsr{1}$, $v_n \to 0$ strongly in $\Lp{q}$ and a.e. in $\R^N$. Thus
	\begin{align*}
	%\label{E7}
	\normLp{\diff u_n}{2}^2 = \normLp{\diff u}{2}^2 + \normLp{\diff v_n}{2}^2 + o_n(1)
	\end{align*}
	and also, by the Brezis-Lieb Lemma \cite{BrezisLieb1983},
	\begin{align}\label{E8}
	\normLp{u_n}{2^*}^{2^*} = \normLp{u}{2^*}^{2^*} + \normLp{v_n}{2^*}^{2^*}  + o_n(1).
	\end{align}
	In particular,
	\begin{align}\label{E9}
	F_{\mu}(u_n) = F_{\mu}(u)  + F_{\mu}(v_n)  + o_n(1).
	\end{align}
	and
	\begin{align}\label{E10}
	Q_{\mu}(u_n) = Q_{\mu}(u)  + Q_{\mu}(v_n)  + o_n(1).
	\end{align}
	Here again we distinguish the two cases 
	$$ \mbox{either} \quad \mbox{(i)} \quad \normLp{v_n}{2^*}^{2^*} \to 0 \quad \mbox{or} \quad \mbox{(ii)} \quad \normLp{v_n}{2^*}^{2^*} \to \ell >0.$$ 
	Assuming that (ii) holds, and since $Q_{\mu}(u) =0$, we deduce from \eqref{E10} that
	$\normLp{v_n}{2^*}^{2^*} = \normLp{\diff v_n}{2}^2 + o_n(1)$. Then, reasoning as in Step 2, it follows that   
	$$F_{\mu}(v_n) \geq \dfrac{\mathcal{S}^{\frac{N}{2}}}{N} + o_n(1)$$
	which leads, in view of \eqref{E9}, to 
	\begin{align*}
	%\label{E11}
	F_{\mu}(u_n) \geq  F_{\mu}(u) + \dfrac{\mathcal{S}^{\frac{N}{2}}}{N} + o_n(1).
	\end{align*}
	At this point, recording from \cref{maps} that $c \mapsto m(c)$ is non increasing, using $Q_{\mu}(u) = 0$ and since, by property of the weak limit, $\normLp{u}{2}^2 \leq c$, we get that 
	$$F_{\mu}(u) \geq m\(\normLp{u}{2}^2\) \geq m(c).$$
	Thus, $F_{\mu}(u_n) \to M^0(c)$ satisfies
	\begin{equation*}
	F_{\mu}(u_n) \geq m(c) + \dfrac{\mathcal{S}^{\frac{N}{2}}}{N} + o_n(1)
	\end{equation*}
	which contradicts our assumption on $M^0(c)$. 
	
	It remains to show that if (i) holds then $(u_n) \subset \Hsr{1}$ converges strongly. Since (i) holds, we get from \eqref{E8} that $\normLp{u_n}{2^*}^{2^*} \to \normLp{u}{2^*}^{2^*}$.
	Choosing $w = u_n$ in \eqref{E1} we deduce since $u$ is solution to \eqref{E3} that
	$$ \normLp{\diff u_n}{2}^2 - \lambda_n \normLp{u_n}{2}^{2}  - \mu \normLp{u_n}{q}^{q} - \normLp{u_n}{2^*}^{2^*} \to \normLp{\diff u}{2}^2 - \lambda \normLp{u}{2}^{2}  - \mu \normLp{u}{q}^{q} - \normLp{u}{2^*}^{2^*} =0.$$
	Therefore, taking into account that $\normLp{u_n}{q}^{q} \to \normLp{u}{q}^{q}$ due to $(u_n) \subset \Hsr{1}$ and since $\lambda_n \to \lambda$, we obtain that
	\begin{align*}
	\normLp{\diff u_n}{2}^2 - \lambda \normLp{u_n}{2}^{2}  \to \normLp{\diff u}{2}^2 - \lambda \normLp{u}{2}^{2}.
	\end{align*}
	By $\lambda <0$ (since $u$ is non trivial), see \cref{lemma:QL} we conclude that $u_n \to u$ strongly in $\Hsr{1}$. At this point the proposition is proved.
\end{proof}

\section{The proof of \cref{prop:44}} \label{Section5}

As announced in the Introduction to show that \cref{prop:44} holds we shall rely on \cref{prop:6} and \cref{prop:5}.

\subsection{The proof of \cref{prop:6}} \medskip

We shall proceed into three steps.

\medbreak
\noindent \textbf{Step 1: } \textit{For any $c \in (0, c_0)$, it holds that
	\begin{align*}
	M(c) \geq \inf_{u \in \slm} F_{\mu}(u).
	\end{align*}}
\medbreak

Let $h \in \Gamma(c)$. We have $h(0) \in V(c) \cap \{u : F_{\mu}(u) <0 \}$ and thus, in view of \cref{lemma:15}\ref{point:lemma15.3},  $h(0) \in W(c)$ or equivalently $s_{h(0)}^+  > 1$. Since $h \in \Gamma(c)$, we also have that $F_{\mu}(h(t)) \leq 2m(c) < m(c)$ for $t$ large enough. Thus, from \cref{lemma:15}\ref{point:lemma15.4}, we get that $h(t) \not \in W(c)$ for $t$ large enough or equivalently that $s_{h(0)}^+  <1$ for such $t >0$.  By the continuity of $h$ and of $u \mapsto s_u^+$, see \cref{lemma:14}\ref{point:lemma14:3}, we deduce that there exists a $t_0 >0$ such that $s_{h(t_0)}^+ = 1$, namely such that $h(t_0) \in \partial W(c)$. Thus we have that
\begin{align*}
M(c) \geq \inf_{u \in \partial W(c)} F_{\mu}(u) = \inf_{u \in \slm} F_{\mu}(u)
\end{align*}
due to \cref{lemma:15}\ref{point:lemma15.2}.

\medbreak
\noindent \textbf{Step 2: } \textit{For any $c \in (0, c_0)$, it holds that
	\begin{align*}
	\inf_{u \in \slm} F_{\mu}(u) \geq \inf_{u \in \slmr} F_{\mu}(u).
	\end{align*}}
\medbreak

For any $u \in \slm$, let $v$ be the Schwarz rearrangement of $|u|$. We claim that $\psi_v(s) \leq \psi_u(s)$ for all $s \geq 0$. Indeed, we have
\begin{align*}
\psi_v(s) &= \dfrac{s^2}{2} \normLp{\diff v}{2}^2 - \dfrac{\mu}{q} s^{\frac{N(q-2)}{2}} \normLp{v}{q}^q - \dfrac{s^{2^*}}{2^*} \normLp{v}{2^*}^{2^*} \\
&\leq \dfrac{s^2}{2} \normLp{\diff u}{2}^2 - \dfrac{\mu}{q} s^{\frac{N(q-2)}{2}} \normLp{u}{q}^q - \dfrac{s^{2^*}}{2^*} \normLp{u}{2^*}^{2^*}
=\psi_u(s).
\end{align*} 
Recording, see \cref{lemma:14}, that $\smu$ is the unique global maximum point for $\psi_u$, we deduce from the above claim that
\begin{align*}
\psi_u(\smu) \geq \psi_u(s_v^+) \geq  \psi_v(s_v^+).
\end{align*}
Since $u \in \slm$, we have that $\smu = 1$ and hence
\begin{align*}
F_{\mu}(u) = \psi_u(1) = \psi_u(\smu) \geq \psi_v(s_v^-) = F_{\mu}(v_{s_v^-}).
\end{align*}
Recording that $v_{s_v^+} \in \slmr$, we deduce that
\begin{align*}
\inf_{u \in \slm} F_{\mu}(u) \geq \inf_{u \in \slmr} F_{\mu}(u).
\end{align*}

\medbreak
\noindent \textbf{Step 3: } \textit{For any $c \in (0, c_0)$, it holds that
	\begin{align*}
	\inf_{u \in \slmr} F_{\mu}(u) \geq M^0(c).
	\end{align*}}
\medbreak

Let $u \in \slmr$ and $s_1 > 0$ be such that $u_{s_1} \in E_c$. Let us consider the map
\begin{align*}
g_u : t \in [0, \infty) \mapsto u_{(1-t)\spu + ts_1} \in S_r(c).
\end{align*}
We have that $g_u \in \mathit{C}([0,\infty), S_r(c))$ and 
\begin{align*}
g_u(0) = u_{\spu} \in \slpr \quad\mbox{and}\quad g_u(1) = u_{s_1} \in E_c.
\end{align*}
Hence, we get $g_u \in \Gamma^0(c)$ and
\begin{align*}
F_{\mu}(u) = \max_{s >0} F_{\mu} (u_s) \geq \max_{t \in [0,\infty)} F_{\mu}(g_u(t)) \geq \inf_{g \in \Gamma^0(c)}  \max_{t \in [0,\infty)} F_{\mu}(g(t)) = M^0(c).
\end{align*}

\begin{proof}[Proof of \cref{prop:6}]
	From \textbf{Steps 1, 2 and 3} we deduce that \cref{prop:6} holds.
\end{proof}

%%%%%%%%%%%%%%%%%%%%%%%%%%%%%%%%%%%%%%%%%%%%%%%%%%%%%%%%%%%%%%%%%%%%%%%%%%%%%%%%%%%%%%%

\begin{remark}\label{various-characterization}
	Trivially, since $\Gamma^0(c) \subset \Gamma(c)$, one has $M(c) \leq M^0(c)$. Thus, from \cref{prop:6} we deduce that 
	$$M^0(c)= M(c) = \inf_{u \in \slm} F_{\mu}(u).$$
\end{remark}

\subsection{The proof of \cref{prop:5}} 

%%%%%%%%%%%%%%%%%%%%%%%%%%%%%%%%%%%%%%%%%%%%%%%%%%%%%%%%%%%%%%%%%%%%%%%%%%%%%%%%%%%%%%%
\begin{lemma} \label{lemma:4}
Let $N \geq 3$.
	For any $c \in (0, c_0)$, the following property holds
	$$M(c) \leq \inf_{h \in \mathcal{G}(c)} \max_{t \in [0,\infty)} F_{\mu}(h(t))$$
	where 
	\begin{align*}
	\mathcal{G}(c) := \Big\{h \in \mathit{C}([0,\infty), \displaystyle \cup_{d \in \[\frac{c}{2},c\]} S(d)) : h(0) \in \mathcal{M}_d \, \mbox{ for some } \, d \in \[\frac{c}{2},c\], \exists t_0 = t_0(h) \mbox{ s.t. } h(t) \in E_c \, \forall t \geq t_0 \Big\}.
	\end{align*}
\end{lemma}
\begin{proof}
	Let any $h \in \mathcal{G}(c)$. We define the function
	\begin{align*}
	t \mapsto \theta(t) := \sqrt{\dfrac{\normLp{h(t)}{2}^2}{c}}.
	\end{align*}
	Note that $\theta$ is the continuous function from $[0, \infty)$ into $\R$ and $\theta(t) \leq 1$ for all $t$.
	Now we set 
	\begin{align}
	g(t)(x) := \theta(t)^{\frac{N}{2}-1} h(t)(\theta(t) x).
	\end{align}
	By direct computations, we obtain that
	\begin{align*}
	\normLp{g(t)}{2}^2 &= \dfrac{1}{[\theta(t)]^2} \normLp{h(t)}{2}^2 = c, &\normLp{\nabla g(t)}{2}^2 &= \normLp{\nabla h(t)}{2}^2, \\
	\normLp{g(t)}{q}^q &= [\theta(t)]^{\(\frac{N}{2}-1\)q - N} \normLp{h(t)}{q}^q, &\normLp{g(t)}{2^*}^{2^*} &= \normLp{h(t)}{2^*}^{2^*}.
	\end{align*}
	Hence,
	\begin{align*}
	F_{\mu} (g(t)) &= \dfrac{1}{2} \normLp{\diff g(t)}{2}^2 - \dfrac{\mu}{q} \normLp{g(t)}{q}^q - \dfrac{1}{2^*}\normLp{g(t)}{2^*}^{2^*} \\
	&= \dfrac{1}{2} \normLp{\nabla h(t)}{2}^2 - \dfrac{\mu}{q}  [\theta(t)]^{\(\frac{N}{2}-1\)q - N} \normLp{h(t)}{q}^q - \dfrac{1}{2^*}\normLp{h(t)}{2^*}^{2^*}\\
	&\leq \dfrac{1}{2} \normLp{\nabla h(t)}{2}^2 - \dfrac{\mu}{q} \normLp{h(t)}{q}^q - \dfrac{1}{2^*}\normLp{h(t)}{2^*}^{2^*} = F_{\mu} (h(t))
	\end{align*}
	due to $\theta(t) \leq 1$ for all $t$ and $\(\dfrac{N}{2}-1\)q - N < 0$. 
	Noting that $F_{\mu}(g(0)) \leq F_{\mu}(h(0)) <0$ and that $\normLp{\nabla g(0)}{2}^2 = \normLp{\nabla h(0)}{2}^2 < \rho_0$ we deduce that $g(0) \in V(c) \cap \{u : F_{\mu}(u) < 0 \}$ and hence that  $g \in \Gamma(c)$. At this point the lemma is proved.
\end{proof}
%%%%%%%%%%%%%%%%%%%%%%%%%%%%%%%%%%%%%%%%%%%%%%%%%%%%%%%%%%%%%%%%%%%%%%%%%%%%%%%%%%%%%%%

%%%%%%%%%%%%%%%%%%%%%%%%%%%%%%%%%%%%%%%%%%%%%%%%%%%%%%%%%%%%%%%%%%%%%%%%%%%%%%%%%%%%%%%
Let $u_{\varepsilon}$ be an extremal function for the Sobolev inequality in $\Rn$ defined by
\begin{align}\label{Def-extremal}
u_{\varepsilon}(x) := \dfrac{[N(N-2)\varepsilon^2]^{\frac{N-2}{4}}}{[\varepsilon^2+|x|^2]^{\frac{N-2}{2}}}, \quad \varepsilon > 0, \quad x\in\Rn.
\end{align}
Let $\xi \in C_0^{\infty}(\R^N)$ be a radially non-increasing cut-off function with $\xi \equiv 1$ in $B_1$, $\xi \equiv 0$ in $\R^N \backslash B_2$.
Setting $U_{\varepsilon}(x) = \xi(x) u_{\varepsilon}(x)$ we shall prove the following lemma.

%%%%%%%%%%%%%%%%%%%%%%%%%%%%%%%%%%%%%%%%%%%%%%%%%%%%%%%%%%%%%%%%%%

\begin{lemma}\label{lemma:7L}
Let $N \geq 3$
	and $u \in H^1(\R^N)$ be a nonnegative function. For every 
	$\varepsilon >0$  and every  $t > 0$ we have
	\begin{align*}
	F_{\mu} (u + t U_{\varepsilon} ) & \leq  F_{\mu}(u) +  t \intrn \diff u(x) \cdot\diff U_{\varepsilon}(x) \, dx 
	+\dfrac{t^2}{2} \normLp{\diff U_{\varepsilon}}{2}^2 - \dfrac{\mu t^{q}}{q} \normLp{U_{\varepsilon} }{q}^{q} - \dfrac{t^{2^*}}{2^*} \normLp{U_{\varepsilon} }{2^*}^{2^*}.
	\end{align*}
\end{lemma}
\begin{proof}	
	We have, for any $\varepsilon >0$ and any $t>0$,
	\begin{align*}
	\normLp{\diff (u + t U_{\varepsilon})}{2}^2 = \normLp{\diff u}{2}^2 + 2 t \intrn \diff u(x) \cdot \diff U_{\varepsilon}(x) \, dx  +   t^2 \normLp{\diff U_{\varepsilon}}{2}^2.
	\end{align*}
	Also, since both $u \in H^1(\R^N)$ and $U_{\var}$ are non negative, 
	\begin{align*}
	\normLp{u + t U_{\varepsilon} }{2^*}^{2^*} & \geq \normLp{u}{2^*}^{2^*} + t^{2^*} \normLp{U_{\varepsilon} }{2^*}^{2^*}
	% \label{eqn:4.1L}
	\end{align*}
	and
	\begin{align*}
	\normLp{u + t U_{\varepsilon} }{q}^{q} & \geq \normLp{u}{q}^{q} + t^{q} \normLp{U_{\varepsilon} }{q}^{q}.
	% \label{eqn:4.2}
	\end{align*}
	Therefore, we obtain that
	\begin{align*}
	\begin{split}
	F_{\mu} (u + t U_{\varepsilon} ) 
	&\leq \dfrac{1}{2} \[ \normLp{\diff u}{2}^2 + 2 t \intrn \diff u(x) \cdot \diff U_{\varepsilon}(x) \, dx  +   t^2 \normLp{\diff U_{\varepsilon}}{2}^2 \] \\
	&- \dfrac{\mu}{q} \[\normLp{u}{q}^{q} + t^{q} \normLp{U_{\varepsilon} }{q}^{q} \] - \dfrac{1}{2^*} \[\normLp{u}{2^*}^{2^*} + t^{2^*} \normLp{U_{\varepsilon} }{2^*}^{2^*} \]\\
	& = F_{\mu}(u) + t \intrn \diff u(x) \cdot \diff U_{\varepsilon}(x) \, dx 
	+ \dfrac{t^2}{2} \normLp{\diff U_{\varepsilon}}{2}^2 - \dfrac{\mu t^{q}}{q} \normLp{U_{\varepsilon} }{q}^{q} - \dfrac{t^{2^*}}{2^*} \normLp{U_{\varepsilon} }{2^*}^{2^*}.
	\end{split}
	\end{align*}
\end{proof}

%%%%%%%%%%%%%%%%%%%%%%%%%%%%%%%%%%%%%%%%%%%%%%%%%%%%%%%%%%%%%%%%%%%%%%%%%%%%%%%%%%%%%%%

%%%%%%%%%%%%%%%%%%%%%%%%%%%%%%%%%%%%%%%%%%%%%%%%%%%%%%%%%%%%%%%%%%%%%%%%%%%%%%%%%%%%%%%

From now and for the rest of this section, we fix a sequence  $(\var_n) \subset \Rp $ such that $\var_n \to 0$.

\begin{lemma}\label{lemma:8L}
Let $N \geq 3$.
	There exists $0 <t_0 < t_1 < \infty$  such that, for any sequence $(u_n) \subset H^1(\R^N)$ satisfying 
	\begin{equation}\label{control}
	\int_{\R^N }\nabla u_n(x) \cdot \nabla U_{\var_n}(x) \, dx  \leq 1, \quad \forall n \in \N,
	\end{equation}
	setting 
	\begin{align*}
	I_n(t):=  t \int_{\R^N }\nabla u_n(x) \cdot \nabla U_{\var_n}(x) \, dx +
	\dfrac{t^2}{2} \normLp{\diff U_{\varepsilon}}{2}^2 - \dfrac{\mu t^{q}}{q} \normLp{U_{\varepsilon} }{q}^{q} - \dfrac{t^{2^*}}{2^*} \normLp{U_{\varepsilon} }{2^*}^{2^*},
	\end{align*}
	we have, for any $n \in \N$ large enough, 
	\begin{enumerate}[label=(\roman*)]
		\item\label{point:lemma8L:1} 
		if $I_n(t) \geq \frac{1}{2N}\mathcal{S}^{\frac{N}{2}}$
		then necessarily $t \geq t_0$,
		\item\label{point:lemma8L:2} $I_n(t) \leq 2 m(c)$ 
		for any $t \geq t_1$.
	\end{enumerate}
\end{lemma}
\begin{proof}
	Observe that
	\begin{align*}
	I_n(t) \leq t \int_{\R^N }\nabla u_n(x) \cdot \nabla U_{\var_n}(x) \, dx + \dfrac{t^2}{2} \normLp{\diff U_{\varepsilon}}{2}^2.
	\end{align*}
	We have that $\normLp{\diff U_{\varepsilon}}{2}^2 \to \cS^{\frac{N}{2}} >0$, see \cref{estimates-U}. Thus in view of \eqref{control}, if $t \to 0$ then $I_n(t) < \frac{1}{2N}\mathcal{S}^{\frac{N}{2}}$. Hence, there exists $t_0 > 0$ such that if  $I_n(t) \geq \frac{1}{2N}\mathcal{S}^{\frac{N}{2}}$ then necessarily $t \geq t_0$ and point \ref{point:lemma8L:1} holds.
	We also have 
	\begin{align*}
	I(t) \leq t \int_{\R^N }\nabla u_n(x) \cdot \nabla U_{\var_n}(x) \, dx  +  \dfrac{t^2}{2} \normLp{\diff U_{\varepsilon}}{2}^2 - \dfrac{t^{2^*}}{2^*} \normLp{U_{\varepsilon} }{2^*}^{2^*}
	\end{align*}
	with $\normLp{\diff U_{\varepsilon}}{2}^2 \to \cS^{\frac{N}{2}} >0$ and $\normLp{U_{\varepsilon} }{2^*}^{2^*} \to \cS^{\frac{N}{2}} >0$, see \cref{estimates-U}. Thus, in view of \eqref{control}, there exists a $t_1 >0$ such that
	$I_n(t) \leq 2 m(c)$, for all $t \geq t_1$, if $n \in \N$ is large enough. Thus point \ref{point:lemma8L:2} also holds.
\end{proof}

%%%%%%%%%%%%%%%%%%%%%%%%%%%%%%%%%%%%%%%%%%%%%%%%%%%%%%%%%%%%%%%%%%%%%%%%%%%%%%%%%%%%%%%

%%%%%%%%%%%%%%%%%%%%%%%%%%%%%%%%%%%%%%%%%%%%%%%%%%%%%%%%%%%%%%%%%%%%%%%%%%%%%%%%%%%%%%%

We define by $\mathcal{M}_c^0$ the set of elements of $\mathcal{M}_c$ which have the properties guarantee by \cref{prop:1}.

\begin{lemma}\label{lemma:9LL}
Let $N \geq 3$, $c \in (0,c_0)$ and $u_c \in \mathcal{M}_c^0$. For any $\varepsilon >0$  there exists a $y_{\varepsilon} \in \R^N$ such that
	\begin{align}\label{keycontrolLL}
	2 \int_{\R^N}u_c( x - y_{\varepsilon}) U_{\varepsilon}(x) \, dx \leq t_1  \normLp{U_{\varepsilon}}{2}^2
	\end{align}
	where $t_1 >0$ is provided by \cref{lemma:8L} and
	\begin{align}\label{keycontrolLLL}
	\int_{\R^N} \nabla u_c( x - y_{\varepsilon}) \cdot \nabla U_{\varepsilon}(x) \, dx \leq  \normLp{U_{\varepsilon}}{2}^2.
	\end{align}
\end{lemma}
\begin{proof}
	Since $u_c \in \mathcal{M}_c^0$ is a radial, non-increasing function, we know from \cite[Radial Lemma A.IV]{BerestyckiLions1983} that
	\begin{equation}\label{decay-rate}
	|u_c(z)| \leq C(N) |z|^{- \frac{N}{2}} \sqrt{c}, \quad \forall |z| \geq 1,
	\end{equation}
	and thus, 
	\begin{align}\label{decreaseL}
	\int_{\R^N}u_c( x - y) U_{\varepsilon}(x)\, dx &\leq C(N) \sqrt{c} \int_{\R^N}  |x-y|^{- \frac{N}{2}}   U_{\varepsilon}(x)\, dx.
	\end{align}
	Using that the function $U_{\varepsilon}$ is compactly supported in $B_2$, we have that, for $|y|$ large enough,
	\begin{align*}
	\int_{\R^N}|x-y|^{- \frac{N}{2}}U_{\varepsilon}(x) dx \leq \int_{B_2} \Big| \frac{y}{2}\Big|^{- \frac{N}{2}}U_{\varepsilon}(x) dx.
	\end{align*}
	At this point, we deduce from \eqref{decreaseL} that, for $|y|$ large enough,
	\begin{align*}
	\int_{\R^N}u_c( x - y) U_{\varepsilon}(x)\, dx \leq C(N) \sqrt{c} \, \Big|\dfrac{y}{2}\Big|^{- \frac{N}{2}} \int_{\R^N}    U_{\varepsilon}(x)\, dx.
	\end{align*}
	Using H\"{o}lder's inequality, we obtain that $\normLp{U_{\varepsilon}}{1} \leq |B_2|^{\frac{1}{2}} \normLp{U_{\varepsilon}}{2}$ and \eqref{keycontrolLL} follows.
	Now, since for any $y \in \R^N,$ $u_c(\cdot-y) \geq0 $ is solution to the equation 
	\begin{align}
	-\laplace u - \mu \abs{u}^{q-2} u - \abs{u}^{2^*-2} u = \lambda_c u \quad \mbox{in } \Rn,
	\end{align}
	for some $\lambda_c <0$,
	we have that
	\begin{align}
	\begin{split}
	\intrn \diff u_c(x -y) \cdot \diff U_{\varepsilon}(x) \, dx \leq 
	& \, \mu \intrn |u_c(x-y)|^{q -1}   U_{\varepsilon}(x) \, dx 
	+ \intrn |u_c(x -y)|^{2^* -1}   U_{\varepsilon}(x) \, dx.
	\end{split}
	\end{align}
	Since, both $|u_c(z)|^{q-1} \leq |u_c(z)|$ and $|u_c(z)|^{2^*-1} \leq |u_c(z)|$ for $|z|$ large, see \eqref{decay-rate}, reasoning as in the proof of \eqref{keycontrolLL} we readily check that \eqref{keycontrolLLL} also holds.
\end{proof}

%%%%%%%%%%%%%%%%%%%%%%%%%%%%%%%%%%%%%%%%%%%%%%%%%%%%%%%%%%%%%%%%%%%%%%%%%%%%%%%%%%%%%%%
Now we define the sequence $(c_n) \in \displaystyle \[\frac{c}{2}, c\)$ as follows
\begin{equation}\label{definitioncn}
c_n := c - 2 t_1^2 \normLp{U_{\varepsilon_n}}{2}^2
\end{equation}
where $t_1>0$ is given in \cref{lemma:8L}. Clearly, $c_n \to c$ as $n \to \infty$. Now, for each $n \in \N $, we fix a $u_{c_n} \in \mathcal{M}_{c_n}^0$. 

\begin{lemma}\label{lemma:9LLL}
	Under the setting introduced above, for any $n \in \N$ large enough, there exists a $y_n \in \R^N$ such that
	\begin{equation}\label{control-from-above}
	c_n \leq \normLp{ u_{c_n}(\cdot - y_{n}) + t U_{\varepsilon_n}}{2}^2 \leq c, \quad \forall t\in [0,t_1],
	\end{equation}
	and
	\begin{equation}\label{control-from-above-bis}
	\int_{\R^N} \nabla u_{c_n}( x - y_{n}) \cdot \nabla U_{\varepsilon_n}(x) \, dx \leq \max \{1, \normLp{ U_{\varepsilon_n}}{2}^2 \}.
	\end{equation}
\end{lemma}
\begin{proof}
	For any $n \in \N$, according to \cref{lemma:9LL}, we can choose a $y_n \in \R^N$ such that
	\begin{align*}
	2 \int_{\R^N}u_{c_n}( x - y_{n}) U_{\varepsilon_n}(x) \, dx \leq t_1 \normLp{ U_{\varepsilon_n}}{2}^2.
	\end{align*}
	Now, for $t \in [0, t_1]$, we have, 
	\begin{align*}
	\normLp{ u_{c_n}(\cdot - y_n) + t U_{\varepsilon_n}}{2}^2 &= \normLp{ u_{c_n}(\cdot - y_n)}{2}^2 + 2t \int_{\R^N} u_{c_n}(x - y_n) U_{\varepsilon_n}(x) \,dx + t^2 \normLp{ U_{\varepsilon}}{2}^2 \\
	& \leq c_n + 2t_1 \int_{\R^N} u_{c_n}(x - y_n) U_{\varepsilon_n}(x) \, dx + t_1^2 \normLp{ U_{\varepsilon_n}}{2}^2 \\
	&\leq c_n +  2 t_1^2 \normLp{ U_{\varepsilon_n}}{2}^2 = c
	\end{align*}
	where for the last inequality we have used the definition of $c_n$ given in \eqref{definitioncn}. The first inequality in \eqref{control-from-above} is obvious by the positivity of $u_{c_n}$ and $U_{\varepsilon_n}$. Finally note that \eqref{control-from-above-bis} directly holds in view of \eqref{keycontrolLLL} and since $\normLp{ U_{\varepsilon_n}}{2}^2 \to 0$ as $n \to \infty.$
\end{proof}

%%%%%%%%%%%%%%%%%%%%%%%%%%%%%%%%%%%%%%%%%%%%%%%%%%%%%%%%%%%%%%%%%%%%%%%%%%%%%%%%%%%%%%%

\begin{lemma}\label{def-gamma-n}
	Under the setting introduced above,  we define,
	\begin{align*}  
	\gamma_n(t) =
	\begin{cases} 
	\displaystyle u_{c_n}(\cdot - y_n) + t U_{\var_n} \quad \mbox{if} \quad t \in [0,t_1],\\
	\displaystyle \gamma_n(t_1) \quad \mbox{if} \quad t \geq t_1.
	\end{cases}
	\end{align*}
	Then $\gamma_n \in \mathcal{G}(c)$, for any $n \in \N$ large enough. In addition
	\begin{equation}\label{key}
	M(c) \leq \max_{t \in [0, \infty)} F_{\mu}(\gamma_n(t)) \leq \max \Big\{ \max_{t \in [t_0,t_1]}F_{\mu}(\gamma_n(t)), \,
	m(c) + \frac{2}{3N}\mathcal{S}^{\frac{N}{2}} \Big\}.
	\end{equation}
\end{lemma}

\begin{proof}
	Let $n \in \N$ be arbitrary fixed, large. To show that $\gamma_n \in \mathcal{G}(c)$ we first observe that $\gamma_n \in C([0, \infty), H^1(\R^N)).$ Also, by \eqref{control-from-above}, we get that
	$$ \gamma_n(t) \subset \bigcup_{d \in \[\frac{c}{2}, c\]}S(d)$$
	since $c_n \to c$. Now, note that by \cref{lemma:7L}
	\begin{align} \label{star1}
	F_{\mu} (\gamma_n(t) )  \leq  m(c_n) +  t \intrn \diff u_{c_n}(\cdot -y_n) \cdot\diff U_{\varepsilon_n}(x) \, dx 
	+\dfrac{t^2}{2} \normLp{\diff U_{\varepsilon_n}}{2}^2 - \dfrac{\mu t^{q}}{q} \normLp{U_{\varepsilon_n} }{q}^{q} - \dfrac{t^{2^*}}{2^*} \normLp{U_{\varepsilon_n} }{2^*}^{2^*}.
	\end{align}
	In view of \eqref{control-from-above-bis}   we can apply \cref{lemma:8L} to deduce that, for $t \geq t_1$
	\begin{equation}\label{star2}
	F_{\mu}(\gamma_n(t)) \leq m(c_n) + 2 m(c) \leq 2 m(c).
	\end{equation}
	We conclude that $\gamma_n \in \mathcal{G}(c)$ for any $n \in \N$ large enough. In particular the first inequality in \eqref{key} holds because of \cref{lemma:4}. Now, considering again \eqref{star1} and recording that \eqref{control-from-above-bis} apply, we deduce from \cref{lemma:8L} that, if $t \in [0,t_0]$,
	\begin{equation}\label{star3}
	F_{\mu}(\gamma_n(t)) \leq m(c_n) + \frac{1}{2N} \mathcal{S}^{\frac{N}{2}} \leq m(c) + \frac {2}{3N} \mathcal{S}^{\frac{N}{2}}
	\end{equation}
	since $c_n \to c$. Gathering \eqref{star2} and \eqref{star3} we see that the second inequality in \eqref{key} holds.
\end{proof}

%\begin{remark}\label{frombelow1}
%Combining \cref{lemma:4}, \cref{lemma:8L} and \cref{lemma:9LL} we have that, for $n \in \N$ large enough,
%\begin{align}\label{Icontrol}
%		M(c) \leq \max_{t \in [t_0,t_1]}F_{\mu} (u_{c_n}(\cdot- y_n) + t U_{\varepsilon_n}).
%	\end{align}
%\end{remark}

%%%%%%%%%%%%%%%%%%%%%%%%%%%%%%%%%%%%%%%%%%%%%%%%%%%%%%%%%%%%%%%%%%%%%%%%%%%%%%%%%%%%%%%
\begin{proof}[Proof of \cref{prop:5}]
	We assume the setting above. In view of \cref{def-gamma-n} to show that
	$$M(c) < m(c) + \frac{1}{N}\mathcal{S}^{\frac{N}{2}}$$
	it suffices to show that
	$$
	\max_{t \in [t_0,t_1]}F_{\mu} (u_{c_n}(\cdot- y_n) + t U_{\varepsilon_n}) < m(c) + \frac{1}{N}\mathcal{S}^{\frac{N}{2}}.
	$$
	From  \cref{lemma:6}, \cref{lemma:7L}, \cref{lemma:9LLL} and the definition of $c_n$ given in \eqref{definitioncn}  we can write
	\begin{align*}
	\begin{split}
	&\max_{t \in [t_0,t_1]}F_{\mu} (u_{c_n}(\cdot -y_n) + t U_{\varepsilon_n} )  \\
	\leq \, &m(c_n) + \max_{t \in [t_0,t_1]} \[ t \normLp{ U_{\varepsilon}}{2}^2 + \dfrac{t^2}{2} \normLp{\diff U_{\varepsilon}}{2}^2  - \dfrac{\mu t^{q}}{q} \normLp{ U_{\varepsilon}}{q}^q - \dfrac{t^{2^*}}{2^*} \normLp{ U_{\varepsilon}}{2^*}^{2^*} \]
	\\
	\leq \, &m(c) + d(c  - c_n) + t_1 \normLp{ U_{\varepsilon}}{2}^2 
	+ \max_{t \in [t_0,t_1]} \[\dfrac{t^2}{2} \normLp{\diff U_{\varepsilon}}{2}^2 - \dfrac{\mu t^{q}}{q} \normLp{ U_{\varepsilon}}{q}^q - \dfrac{t^{2^*}}{2^*} \normLp{ U_{\varepsilon}}{2^*}^{2^*} \] \\
	\leq \, &m(c) +   2t_1^2 d  \normLp{U_{\varepsilon_n}}{2}^2  + t_1 \normLp{ U_{\varepsilon}}{2}^2 
	+ \max_{t \in [t_0,t_1]} \[\dfrac{t^2}{2} \normLp{\diff U_{\varepsilon}}{2}^2 - \dfrac{\mu t^{q}}{q} \normLp{ U_{\varepsilon}}{q}^q - \dfrac{t^{2^*}}{2^*} \normLp{ U_{\varepsilon}}{2^*}^{2^*} \].
	\end{split} 
	\end{align*}
	To complete the proof it suffices to show, for $n \in \N$ sufficiently large,
	\begin{align*}
	J_n:=  (2t_1^2 d + t_1)  \normLp{U_{\varepsilon_n}}{2}^2 + \max_{t \in [t_0,t_1]} \[\dfrac{t^2}{2} \normLp{\diff U_{\varepsilon}}{2}^2 - \dfrac{\mu t^{q}}{q} \normLp{ U_{\varepsilon}}{q}^q - \dfrac{t^{2^*}}{2^*} \normLp{ U_{\varepsilon}}{2^*}^{2^*} \]
	< \frac{1}{N}\mathcal{S}^{\frac{N}{2}}.
	\end{align*}
	In this aim first note that
	$$J_n \leq  (2t_1^2 d + t_1)  \normLp{U_{\varepsilon_n}}{2}^2 + \max_{t >0} \[\dfrac{t^2}{2} \normLp{\diff U_{\varepsilon}}{2}^2  - \dfrac{t^{2^*}}{2^*} \normLp{ U_{\varepsilon}}{2^*}^{2^*} \] - \dfrac{\mu t_0^{q}}{q} \normLp{ U_{\varepsilon}}{q}^q. $$
	But it holds, in view of the estimates of \cref{estimates-U}\ref{point:estimates-U:1},
	that
	\begin{equation*}
	%\label{SecondX1}
	\max_{t >0} \[\dfrac{t^2}{2} \normLp{\diff U_{\varepsilon}}{2}^2  - \dfrac{t^{2^*}}{2^*} \normLp{ U_{\varepsilon}}{2^*}^{2^*}  \]  = 
	\frac{1}{N}\mathcal{S}^{\frac{N}{2}} + O(\var_n^{N-2}).
	\end{equation*}
	Summarizing, the proof of  \cref{prop:5} will be completed if we manage to show that, for $n \in \N$ large enough,
	\begin{equation*}
	%\label{SecondX2}
	(2t_1^2 d + t_1)  \normLp{U_{\varepsilon_n}}{2}^2  - \dfrac{\mu t_0^{q}}{q} \normLp{ U_{\varepsilon}}{q}^q 
	+  O(\varepsilon_n^{N-2})  < 0.
	\end{equation*}
	At this point we distinguish two cases. \medskip
	
	\noindent\textbf{\underline{Case 1} $N \geq 5$}: By \cref{estimates-U} we have that, for some $K_1 >0, K_2 >0$,
	as $n \to \infty$,
	$$\normLp{U_{\varepsilon_n}}{2}^2 = K_1 \varepsilon_n^2  + o(\varepsilon_n^{N-2}) \quad \mbox{and} \quad  \normLp{U_{\varepsilon_n}}{q}^q = K_2 \varepsilon_n^{N - \frac{(N-2)q}{2}}  + o(\varepsilon_n^{N - \frac{(N-2)q}{2}}).$$
	Thus, for some constants, $\tilde{K}_1 >0$, $\tilde{K}_2 >0$, for $n \in \N$ sufficiently large,
	\begin{align*}
	(2t_1^2 d + t_1)  \normLp{U_{\varepsilon_n}}{2}^2 - \dfrac{\mu t_0^{q}}{q} \normLp{ U_{\varepsilon}}{q}^q  + O(\varepsilon_n^{N-2})  =
	\tilde{K}_1 \varepsilon_n^2  - \tilde{K}_2 \, \varepsilon_n^{N - \frac{(N-2)q}{2}}  + O(\varepsilon_n^{N-2})
	\leq - \frac{\tilde{K}_2}{2} \, \varepsilon_n^{N - \frac{(N-2)q}{2}} <0
	\end{align*}
	since $N - \frac{(N-2)q}{2} < \min \{N-2, 2 \} \Longleftrightarrow q >2.$  \medskip
	
	\noindent\textbf{\underline{Case 2} $N =4 $}: By \cref{estimates-U} we have that
	as $n \to \infty$, for some $K_4 >0$,
	$$ \normLp{U_{\varepsilon_n} }{2}^{2} =  \omega\var_n^2 |\log \var_n| + O(\var_n^2) \quad \mbox{and} \quad  \normLp{U_{\varepsilon_n} }{q}^{q} =  K_4 \var_n^{4-q} + o(\var_n^{4-q}).$$
	Thus, for some constants, $\tilde{K}_3 >0$, $\tilde{K}_4 >0$, for $n \in \N$ sufficiently large, 
	\begin{align*}
	(2t_1^2 d + t_1) \normLp{U_{\varepsilon_n}}{2}^2 - \dfrac{\mu t_0^{q}}{q} \normLp{ U_{\varepsilon}}{q}^q  + O(\varepsilon_n^{2})  =
	\tilde{K}_3 \, \varepsilon_n^2  |\log \var_n| - \tilde{K}_4 \, \varepsilon_n^{4-q}  + O(\varepsilon_n^{2})
	\leq - \frac{\tilde{K}_4}{2} \, \varepsilon_n^{4-q} <0
	\end{align*}
	since $4-q < 2$. \medskip	
	In view of Cases 1 and 2 we  deduce that the conclusion of \cref{prop:5} holds if $N \geq 4$.
\end{proof}

\begin{remark}\label{lien-tarantello}
Our analysis of the interaction between a solution characterized as a local minima and a suitable family of truncated extremal functions for the Sobolev inequality reminds us of the approach developed by G. Tarantello in \cite{Tarantello92}. However, in \cite{Tarantello92}, the extremal functions are located in a set where the local minima solution takes its greater values. The idea being to prove, through delicate estimates, that this interaction does decrease the mountain pass value of the associated functional with respect to the case where the two supports would be fully disjoint. Here, on the contrary, our construction aims at separating sufficiently the regions where the functions concentrate and to show that the remaining interaction (remember our functions $u_c \in S(c)$ ly on all $\R^N$) can be assumed sufficiently small. 
\end{remark}
%%%%%%%%%%%%%%%%%%%%%%%%%%%%%%%%%%%%%%%%%%%%%%%%%%%%%%%%%%%%%%%%%%%%%%%%%%%%%%%%%%%%%%%

\section{Proofs of main theorems and additional properties }\label{Section6}

\begin{proof}[Proof of \cref{thm-main}]
The proof of \cref{thm-main} follows directly combining \cref{{prop:3}}, \cref{prop:4} and \cref{prop:44}.
\end{proof}

\begin{proof}[Proof of \cref{thm-additional}]
	For $\mu >0$ fixed, let us prove that 
	\begin{equation}\label{limit}
	F_{\mu}(v_c) \to \dfrac{\mathcal{S}^{\frac{N}{2}}}{N} \quad \mbox{as} \quad  c \to 0.
	\end{equation}
	First,  using that $Q_{\mu}(v_c) = 0$,  we can write, using the Gagliardo-Nirenberg inequality (\ref{Gagliardo-Nirenberg-I})
	\begin{align}\label{ajout-louis1}
	\begin{split}
	F_{\mu}(v_c) &= \dfrac{1}{N} \normLp{\diff v_c}{2}^2 - \dfrac{\mu}{q} \Big( 1 - \frac{q \gamma_q}{2^*} \Big) \normLp{u_n}{q}^q \\
	&\geq \dfrac{1}{N} \normLp{\diff v_c}{2}^2 - \dfrac{\mu}{q} C_{N,q}^q \Big( 1 - \frac{q \gamma_q}{2^*} \Big) c^{(1- \gamma_q)q} \, \normLp{\diff v_c}{2}^{q \gamma_q}.
	\end{split}
	\end{align} 
	Since
	\begin{equation}\label{lastequation}
	F_{\mu}(v_c) = M^0(c) < m(c) + \dfrac{\mathcal{S}^{\frac{N}{2}}}{N} \leq \dfrac{\mathcal{S}^{\frac{N}{2}}}{N} 
	\end{equation}
	and $q \gamma_q <2$ we deduce that $(v_c) \subset \Hsr{1}$ is uniformly bounded with respect to $c \in (0, c_0).$ Thus, still using the Gagliardo-Nirenberg inequality, we deduce that  
	\begin{align}\label{ajout-louisx}
	\normLp{v_c}{q}^q \leq C_{N,q} \normLp{v_c}{2}^{q(1-\gamma_q)} \normLp{\diff v_c}{2}^{q\gamma_q} \to 0 \quad\mbox{as } c \to 0. 
	\end{align}
	Hence, recording that $Q_{\mu}(v_c)=0$, we have that
	\begin{align*}
	\ell := \lim_{c\to 0} \normLp{\diff v_c}{2}^2 = \lim_{c\to 0} \normLp{v_c}{2^*}^{2^*} \leq \dfrac{1}{\cS^{\frac{2^*}{2}}} \lim_{c\to 0} \normLp{\diff v_c}{2}^{2^*} = \dfrac{1}{\cS^{\frac{2^*}{2}}} \ell^{\frac{2^*}{2}}.
	\end{align*}
	Therefore, either $\ell = 0$ or $\ell \geq \cS^{\frac{N}{2}}$. We claim that $\ell = 0$  is impossible. Indeed, since  $v_c \in \slm$, we have that
	\begin{align}
	\normLp{\diff v_c}{2}^2 - \mu \gamma_q \normLp{v_c}{q}^q - \normLp{v_c}{2^*}^{2^*} = 0 \label{eqn:lemma16:1}
	\end{align}
	and 
	\begin{align}
	\normLp{\diff v_c}{2}^2 - \mu \gamma_q (q\gamma_q - 1) \normLp{v_c}{q}^q - (2^*-1) \normLp{v_c}{2^*}^{2^*} \leq 0. \label{eqn:lemma16:2}
	\end{align}
	Combining \eqref{eqn:lemma16:1}, \eqref{eqn:lemma16:2} and using the Sobolev inequality, we get 
	\begin{align*}
	\normLp{\diff v_c}{2}^2 \leq \dfrac{2^*-q\gamma_q}{2 - q\gamma_q} \normLp{v_c}{2^*}^{2^*} \leq \dfrac{2^*-q\gamma_q}{2 - q\gamma_q} \dfrac{1}{\cS^{\frac{2^*}{2}}} \normLp{\diff v_c}{2}^{2^*}
	\end{align*}
	proving the claim. 
	
	At this point, in view of (\ref{ajout-louisx}), using that $Q_{\mu}(v_c)=0$, we have 
	\begin{align*}
	M^0(c) = F_{\mu} (v_c) =  \dfrac{1}{N} \normLp{\diff v_c}{2}^2 - \dfrac{\mu}{q} \(1- \dfrac{q\gamma_q}{2^*}\) \normLp{v_c}{q}^q
	= \dfrac{1}{N} \normLp{\diff v_c}{2}^2 + o_c(1)
	\geq \dfrac{\mathcal{S}^{\frac{N}{2}}}{N} + o_c(1),
	\end{align*}
	where $o_c(1) \to 0$ as $c \to 0$.
	Taking into account that $m(c) \to 0$ as $c \to 0$, see \cref{maps}, and that
	\begin{align*}
	M^0(c) < m(c) + \dfrac{\mathcal{S}^{\frac{N}{2}}}{N}
	\end{align*}
	we obtain (\ref{limit}). Clearly also the above proof shows that $\normLp{\diff v_c}{2}^2  \to \mathcal{S}^{\frac{N}{2}}.$ This proves 
	 \ref{point:thm-additional:1}. 
	
	To show that \ref{point:thm-additional:2} holds we start to observe that, since $c_0(\mu) \to \infty$ as $\mu \to 0$, see \cref{playbetweenparameter}, $v_c \in S(c)$ exists for any $\mu \to 0$ sufficiently small. Now, \eqref{ajout-louis1}- \eqref{lastequation} imply that $(v_c ) \subset \Hsr{1}$ is uniformly bounded as $\mu \to 0$ and thus, using the Gagliardo-Nirenberg inequality, we have that 
	\begin{align*}
	%\label{ajout-louis2}
	\mu \normLp{v_c}{q}^q \leq \mu C_{N,q} \normLp{v_c}{2}^{q(1-\gamma_q)} \normLp{\diff v_c}{2}^{q\gamma_q} \to 0 \quad\mbox{as } \mu \to 0. 
	\end{align*}
	From here the rest of the proof is identical to the one of \ref{point:thm-additional:1}.
\end{proof}

\begin{proof}[Proof of \cref{thm-secondary}]
	We use \cite[Theorem 1.6]{Soave2020Sobolevcriticalcase} which in our notation reads as
	
	\begin{theorem}\label{Soaveblowup}
		Under the assumptions of \cref{thm-main}, let $u \in S(c)$ be such that $F_{\mu}(u) < \inf_{\Lambda^+(c)}F_{\mu}.$ Then, if $s_u^+ <1$ and $|x|u \in L^2(\R^N, \C)$, the solution $\phi$ of (\ref{NLS0}) with initial datum $u$ blows-up in finite time.
	\end{theorem}
	Since, by \cref{various-characterization}
	\begin{align*}
			F_{\mu}(v_c) = \inf_{u \in \slm} F_{\mu}(u)
		\end{align*}
		and $\psi_{v_c}(s)$ has a unique global maximum at $s^{\star} = 1$, see \cref{lemma:14}, we have that $v_t := (v_c)_t$ satisfies
		$$ F_{\mu}(v_t) < \inf_{u \in \slm} F_{\mu}(u)$$
		for any $t>1$. Clearly $s_{v_t}^+ = \frac{1}{t} < 1$ for any $t>1$ and also $v_t \to v_c$ as $t \to 1^+$.
		Now, since $\lambda_c <0,$ by a classical decay argument, we have that $v_c$ and thus $v_t$ satisfies
		$|x|v_t \in L^2(\R^N, \C)$. At this point, applying  \cref{Soaveblowup}, we deduce that $v_c \in S(c)$ gives rise to an unstable standing wave.  See for more details \cite[Theorem 1.3]{Soave2020Sobolevcriticalcase} or \cite[Theorem 1.4]{Soave2020}.

\end{proof}

\begin{remark}
	\cref{Soaveblowup}, see also \cite[Theorem 1.13]{Soave2020}, is remarkable because it permits to detect a finite time blow-up occurs just by considering the 1-variable function $\psi_u$. We refer, for earlier results on the link between the variational characterization of a solution and its instability, to the classical paper \cite{BerestyckiCazenave1981}, and to \cite{Lecoz2008} for more recent developments.
\end{remark}

%%%%%%%%%%%%%%%%%%%%%%%%%%%%%%%%%%%%%%%%%%%%%%%%%%%%%%%%%%%%%%%%%%%%%%%%%%

\section{Appendix} \label{Section7}

Let $N \geq 3$,  $u_{\varepsilon}$ be the extremal functions for the Sobolev inequality in $\Rn$ defined in \eqref{Def-extremal}
and $\xi \in C_0^{\infty}(\R^N)$ be a radially non-increasing cut-off function with $\xi \equiv 1$ in $B_1$, $\xi \equiv 0$ in $\R^N \backslash B_2$.

\begin{lemma}\label{estimates-U}
	Setting $U_{\varepsilon}:= \xi u_{\var}$ and denoting by $\omega$ the area of the unit sphere in $\R^N$, we have, for $N \geq 3$, 
	\begin{enumerate}[label=(\roman*)]
		\item \label{point:estimates-U:1} \begin{align*} 
		%\label{eqn:2.18}
			\normLp{\diff U_{\varepsilon}}{2}^2 = \mathcal{S}^{\frac{N}{2}} + O(\varepsilon^{N-2}) \quad \mbox{and} \quad \normLp{U_{\varepsilon} }{2^*}^{2^*} = \mathcal{S}^{\frac{N}{2}} + O(\varepsilon^N),
		\end{align*}
		\item \label{point:estimates-U:3} For some positive constant $K >0$,
		\begin{align*}  
			\normLp{U_{\varepsilon} }{q}^{q} = \displaystyle
			\begin{cases} 
			\displaystyle	K \var^{N - \frac{(N-2)}{2}q} + o(\var^{N - \frac{(N-2)}{2}q}) \quad &\mbox{if} \quad \dfrac{N}{N-2} < q < 2^*, \\
			\displaystyle	\omega \var^{\frac{N}{2}} |\log \var| + O(\var^{\frac{N}{2}}) \quad &\mbox{if} \quad q = \dfrac{N}{N-2}, \\
			\displaystyle	\omega \Big(\int_0^2 \frac{\xi^q(r)}{r^{(N-2)q - (N-1)}} dr \Big) \var^{\frac{N-2}{2}q} + o(\var^{\frac{N-2}{2}q}) \quad &\mbox{if} \quad 1 \leq q < \dfrac{N}{N-2}.	
			\end{cases}
		\end{align*}
	\end{enumerate}
\end{lemma}

\begin{proof}
	The point \ref{point:estimates-U:1} is standard. See, for example, \cite[pages 163-164]{Struwe}. We shall thus concentrate on point \ref{point:estimates-U:3}. We have
	\begin{align*}
		\normLp{U_{\varepsilon} }{q}^{q} = \int_{\R^N} \xi^q (x) \Big( \frac{\var}{\var^2 + |x|^2} \Big)^{\frac{N-2}{2}q} dx.
	\end{align*}
	Passing to radial coordinates, we get
	\begin{equation}\label{A1}
		\normLp{U_{\varepsilon} }{q}^{q} = \omega \, \var^{\frac{N-2}{2}q} \int_0^2 \frac{\xi^q(r) r^{N-1}}{(\var^2 + r^2)^{\frac{N-2}{2}q}} \, dr
	\end{equation}
	that can be decomposed as
	\begin{equation}\label{A2}
		\normLp{U_{\varepsilon} }{q}^{q} = \omega \var^{\frac{N-2}{2}q} \int_0^2 \frac{(\xi^q(r)-1) r^{N-1}}{(\var^2 + r^2)^{\frac{N-2}{2}q}} dr + \omega \var^{\frac{N-2}{2}q} \int_0^2 \frac{ r^{N-1}}{(\var^2 + r^2)^{\frac{N-2}{2}q}} dr := I_1(\var) + I_2(\var).
	\end{equation}
	Since $\xi(r) \equiv 1$ on $[0,1]$, the integral in $I_1(\var)$ is converging and thus $I_1(\var)= O(\var^{\frac{N-2}{2}q}).$ Now, by making a change of variable, we rewrite $I_2(\var)$ as 
	\begin{align*}
		I_2(\var) = \omega \, \var^{N -\frac{N-2}{2}q} \int_0^{\frac{2}{\var}} \frac{ r^{N-1}}{(1 + r^2)^{\frac{N-2}{2}q}} \, dr.
	\end{align*}
	The integral in  $I_2(\var)$  is converging, as $\var \to 0$, to a finite value if and only if $q > \frac{N}{N-2}$. Thus, when $q > \frac{N}{N-2}$, we have that, for some constant $K >0$,
	\begin{align*}
		I_2(\var)= K \var^{N -\frac{N-2}{2}q}  + o(\var^{N -\frac{N-2}{2}q})
	\end{align*}
	and  recording that $I_1(\var)= O(\var^{\frac{N-2}{2}q})$, we have
	\begin{align*}
		\normLp{U_{\varepsilon} }{q}^{q} = I_1(\var) + I_2(\var) = K \var^{N -\frac{N-2}{2}q}  + o(\var^{N -\frac{N-2}{2}q}).
	\end{align*}
	This proves point \ref{point:estimates-U:3} for $\dfrac{N}{N-2} < q < 2^*$.
	Now, assuming that $q = \dfrac{N}{N-2}$, and proceeding as in \eqref{A2}, we get that
	\begin{equation}\label{A3}
		||U_{\varepsilon}||_{L^q(\R^N)}^{q} = \omega \, \var^{\frac{N}{2}} \int_0^2 \frac{(\xi^q(r)-1) r^{N-1}}{(\var^2 + r^2)^{\frac{N}{2}}} \, dr + \omega \var^{\frac{N}{2}} \int_0^2 \frac{ r^{N-1}}{(\var^2 + r^2)^{\frac{N}{2}}} \, dr := I_1(\var) + I_2(\var)
	\end{equation}
	with $I_1(\var)= O(\var^{\frac{N}{2}} )$. Also,
	\begin{align*}
		I_2(\var) = \omega \var^{\frac{N}{2}} \int_0^{\frac{2}{\var}} \frac{ r^{N-1}}{(1 + r^2)^{\frac{N}{2}}} dr = \omega \var^{\frac{N}{2}} \big( |\log \var| + O(1) \big).
	\end{align*}
	Summarizing, we obtain for $q = \dfrac{N}{N-2}$ that
	\begin{align*}
		\normLp{U_{\varepsilon} }{q}^{q} = \omega \var^{\frac{N}{2}} |\log \var| + O(\var^{\frac{N}{2}}).
	\end{align*}
	It remains to study the case $1 \leq q < \dfrac{N}{N-2}$. Under this assumption, we observe that, for all $r>0$,
	\begin{align*}
		\lim_{\var \to 0} \frac{\xi^q(r) r^{N-1}}{(\var^2 + r^2)^{\frac{N-2}{2}q}} = \frac{\xi^q(r) r^{N-1}}{r^{(N-2)q}} = \frac{\xi^q(r)}{r^{(N-2)q - (N-1)}}
	\end{align*}
	and also that, for all $\var >0$, for some constant $D >0$
	\begin{align*}
		\Big| \frac{\xi^q(r) r^{N-1}}{(\var^2 + r^2)^{\frac{N-2}{2}q}} \Big| \leq \frac{D}{r^{(N-2)q - (N-1)}} \in L^1([0,2]).
	\end{align*}
	Thus, from Lebesgue's theorem, we deduce from \eqref{A1} that
	\begin{align*}
		\normLp{U_{\varepsilon} }{q}^{q} = \omega \Big(\int_0^2 \frac{\xi^q(r)}{r^{(N-2)q - (N-1)}} dr \Big) \var^{\frac{N-2}{2}q} + o(\var^{\frac{N-2}{2}q}).
	\end{align*}
	This ends the proof of the lemma.
\end{proof}

%%%%%%%%%%%%%%%%%%%%%%%%%%%%%%%%%%%%%%%%%%%%%%%%%%%%%%%%%%%%%%%%%%%%%%%%%%%%%%%%%%%%%%%

\begin{remark}
	Even if we do not manage to prove that the strict inequality of \cref{prop:5} holds when $N=3$ we have include the treatment of this case in \cref{estimates-U}, hoping that these estimates will be proved useful at some moment.
\end{remark}

\section*{Declarations}
\label{conflict-interest}
\noindent \textbf{Conflict of Interest.} On behalf of all authors, the corresponding author states that there is no conflict of interest.
\smallbreak
\noindent \textbf{Data Availability Statement.} This article has no additional data.

\renewcommand{\bibname}{References}
\bibliographystyle{plain}
\bibliography{References}
\vspace{0.25cm}
\end{document}